\documentclass[11pt,a4paper]{amsart}
\usepackage[top=3.5cm,bottom=3cm,left=2.5cm,right=2.5cm]{geometry}
\usepackage[utf8]{inputenc}
\usepackage[english]{babel}
\usepackage{float}
\usepackage{xspace}
\usepackage{lmodern}
\DeclareFontFamily{U}{wncyr}{}
\DeclareFontShape{U}{wncyr}{m}{n}{<->wncyr10}{}
\DeclareFontShape{U}{wncyr}{m}{it}{<->wncyi10}{}
\DeclareFontShape{U}{wncyr}{m}{sc}{<->wncysc10}{}
\DeclareFontShape{U}{wncyr}{b}{n}{<->wncyb10}{}
\DeclareTextCommand{\guillemotleft}{T1}{%
  {\fontencoding{U}\fontfamily{wncyr}\selectfont\symbol{"3C}}%
}
\DeclareTextCommand{\guillemotright}{T1}{%
  {\fontencoding{U}\fontfamily{wncyr}\selectfont\symbol{"3E}}%
}

\usepackage{amsmath,amsfonts,amsthm,amssymb}
\usepackage{tikz-cd}

\usepackage{graphicx}
\usepackage{tikz}
\usepackage{ifthen}
\usetikzlibrary{arrows,calc,intersections,patterns}
\usepackage{rotating} 
\usepackage{caption}
\usepackage{subcaption}

\usepackage{enumitem}

\definecolor{amaranth}{rgb}{0.9, 0.17, 0.31}
\definecolor{blue(ryb)}{rgb}{0.01, 0.28, 1.0}
\definecolor{darkred}{rgb}{0.55, 0.0, 0.0}

\allowdisplaybreaks[3] 
\usepackage{tikz-cd}
\usepackage{mathtools}
\usepackage{hyperref}

\usepackage{xcolor}
\hypersetup{
  colorlinks   = true,
  urlcolor     = blue,
  linkcolor    = blue(ryb),
  citecolor   = darkred
}

\newtheorem{theorem}{Theorem}[section]
\newtheorem{proposition}[theorem]{Proposition}

\newtheorem{lemma}[theorem]{Lemma}

\theoremstyle{definition}
\newtheorem{definition}[theorem]{Definition}

\theoremstyle{remark}
\newtheorem{remark}[theorem]{Remark}

\numberwithin{equation}{section}
\DeclareRobustCommand{\SkipTocEntry}[5]{}
\usepackage{siunitx}

\usepackage{microtype}

\begin{document}

\title[Invariant Measures and Growth for the 2D Euler Equations]{Construction of High Regularity Invariant Measures for the 2D Euler Equations and Remarks on the Growth of the Solutions}

\author{Mickaël Latocca}

\address{Département de Mathématiques et Applications, Ecole Normale Supérieure -- PSL Research University, 45 rue d'Ulm 75005 Paris, France}

\email{mickael.latocca@ens.fr}

\date{\today}

\maketitle

\begin{abstract}
We consider the Euler equations on the two-dimensional torus and construct invariant measures for the dynamics of these equations, concentrated on sufficiently regular Sobolev spaces so that strong solutions are also known to exist. The proof follows the method of Kuksin in~\cite{Kuksin2004} and we obtain in particular that these measures do not have atoms, excluding trivial invariant measures. Then we prove that almost every initial data with respect to the constructed measures give rise to global solutions for which the growth of the Sobolev norms are at most polynomial. To do this, we rely on an argument of Bourgain. Such a combination of Kuksin's and Bourgain's arguments already appear in the work of Sy~\cite{sy}. We point out that up to the knowledge of the author, the only general upper-bound for the growth of the Sobolev norm to the $2d$ Euler equations is double exponential. 
\end{abstract}

\section{Introduction}

\subsection{The Euler equations} This article is concerned with the incompressible Euler equations posed on the torus $\mathbb{T}^2$ of dimension $2$:
\begin{equation}
    \label{4Euler2}
    \tag{E${}^2$}
    \left\{
    \begin{array}{ccc}
        \partial_tu + u\cdot \nabla u +\nabla p & =& 0 \\
        \nabla \cdot u &=&0 \\
        u(0)&=& u_0 \in H^s(\mathbb{T}^d,\mathbb{R}^d)\,, 
    \end{array}
    \right.
\end{equation}
where $s>0$, $H^s$ stands for the usual Sobolev space and the unknowns are the velocity field $u(t) : \mathbb{T}^2 \to \mathbb{R}^2$ and the pressure $p : \mathbb{T}^2 \to \mathbb{R}$.

We recall that the pressure $p$ can be recovered from $u$ by solving the elliptic problem 
\[
    -\Delta p = \nabla \cdot \left(u \cdot \nabla u\right)\text{ on } \mathbb{T}^2\,.
\]

In dimension $2$, the \textit{vorticity} defined as $\xi \coloneqq \nabla \wedge u$ is a more convenient variable. Taking the rotational in~\eqref{4Euler2}, the latter can be recast as:
\begin{equation}
    \label{4EulerVorticity}
    \left\{
    \begin{array}{ccc}
         \partial _t \xi + u \cdot \nabla \xi &=& 0  \\
         \xi(0) &=& \nabla \wedge u_0\,,
    \end{array}
    \right.
\end{equation}
where $u$ is recovered from $\xi$ \textit{via} the Biot-Savart law: 
\begin{equation}
    \label{4BiotSavart}
    \mathcal{K}(\xi)(t,x)= \frac{1}{2\pi}\int_{\mathbb{T}^2}\frac{\xi(t,y)(x-y)^{\perp}}{|x-y|^2}\,\mathrm{d}y\,,
\end{equation}
where we recall that for any $y=(y_1,y_2)\in\mathbb{T}^2$ we define $y^{\perp}=(-y_2,y_1)$. 

The Cauchy problem for~\eqref{4Euler2} is well-undestood:

\begin{theorem}[Wolibner, \cite{wolibner}]\label{4mainGlobalWolib} Let $s>2$. The Cauchy problem for~\eqref{4Euler2} is globally well-posed in $\mathcal{C}^0(\mathbb{R}_+,H^s(\mathbb{T}^2,\mathbb{R}^2))$.
\end{theorem}

\subsection{Main results}

\subsubsection{Invariant measures for~\eqref{4Euler2}}

In this article we first construct invariant measure for~\eqref{4Euler2} supported at high regularity. 

\begin{theorem}\label{4mainTheo2d3d} Let $s>2$. There exists a measure $\mu_s$ concentrated on $H^s(\mathbb{T}^2,\mathbb{R}^2)$ such that: 
\begin{enumerate}[label=(\textit{\roman*})]
    \item The equation~\eqref{4Euler2} is $\mu_s$-almost-surely globally well-posed in time. 
    \item $\mu_s$ is an invariant measure for~\eqref{4Euler2}.  
    \item $\mu_s$ does not have any atom and satisfies $\mathbb{E}_{\mu_s}[\|u\|_{H^{s}}^2] = C(s) \in (0,\infty)$.   
    \item $\mu_s$ charges large norm data; that is, for any $R>0$, 
    \[
        \mu_s(u \in H^{s}, \; \|u\|_{H^{s}}>R) >0\,.
    \] 
\end{enumerate}
\end{theorem}

\begin{remark} The measure $\mu_s$ depends on $s$ because it is constructed by compactness methods based on invariant measures for~\eqref{4NSvv2}, which contains a regularisation by an hyper-viscous term $(-\Delta)^{s-1}$, thus depending on $s$.   
\end{remark}

The measure constructed by Theorem~\ref{4mainTheo2d3d} satisfies the following properties.  

\begin{theorem}[Properties of the measure]\label{4mainCoro2d}
Let $s>2$ and $\mu_s$ the measure obtained in Theorem~\ref{4mainTheo2d3d}. There exists a continuous increasing function $p : \mathbb{R}_+ \to \mathbb{R}_+$ such that $p(0)=0$ and for every Borel subset $\Gamma \subset \mathbb{R}_+$ there holds:
    \[
        \mu_s \left(u \in H^{s}, \|u\|_{L^2} \in \Gamma\right) \leqslant p(|\Gamma|)\,.
    \]
\end{theorem}

\subsubsection{Remarks on the growth of the Sobolev norms for solutions to~\eqref{4Euler2}}

This work was originally motivated by the growth of the Sobolev norms and the $L^{\infty}$ norm of the vorticity gradient in dimension $2$. Let us recall some known results. Up to the knowledge of the author, the only known technique to establish such bounds is to estimate $\|\mathcal{K}\|_{L^p\to L^p}$, with $\mathcal{K}$ being defined by~\eqref{4BiotSavart}. For example, in order to estimate $\|\xi\|_{H^s}$, one simply applies $\nabla ^s $ to ~\eqref{4EulerVorticity} and obtain bounds of the form:
\[
    \frac{\mathrm{d}}{\mathrm{d}t} \|\nabla ^s \xi (t)\|_{L^2} \lesssim \|K\|_{L^{p}\to L^p} \|\nabla ^s\xi(t)\|_{L^2}^{1+\frac{2}{p}}\,,
\] 
which after optimisation in $p$ leads to the following estimates. 

\begin{theorem}[Beale-Kato-Majda criterion~\cite{bkm}] Let $u$ be a smooth, global solution to~\eqref{4Euler2}. There exists a constant $C=C(u_0)>0$ such that for any $t >0$,
\[ 
    \|u(t)\|_{H^s}, \|\nabla \xi (t)\|_{L^{\infty}} \leqslant Ce^{e^{Ct}}\,.
\]
\end{theorem}

The question is then to estimate how much these norm can \textit{really} grow. Some specific initial data which produce infinite norm inflation in $\mathbb{T}^2$ were exhibited by Bahouri and Chemin in~\cite{bahouriChemin}. More recently, the work of Kiselev and \u{S}ver\'ak achieved the double exponential growth on a disk. On the torus however, and up to the knowledge of the author, no such result is known. Only initial data producing exponential growth are known. This is the content of the result of Zlato\u{s} in~\cite{zlatos}. We summarise these two results in the following theorem.  

\begin{theorem}[Examples of growth estimates, \cite{kiselevSverak,zlatos}] 
Let $U$ denote either the unit disc $D=\{(x,y) \in \mathbb{R}^2, \; x^2+y^2\leqslant 1\}$ or the torus $\mathbb{T}^2$. There exists an initial data $u_0$ which lies in $\mathcal{C}^{\infty}$ in the case of $U=D$ and $\mathcal{C}^{1,\alpha}$ (for some $\alpha \in (0,1)$) if $U=\mathbb{T}^2$ such that the unique associate global solution to the Euler equation~\eqref{4Euler2} constructed in Theorem~\ref{4mainGlobalWolib} satisfies:
\begin{enumerate}[label=(\textit{\roman*})]
    \item If $U=D$, then there exists $C>0$ such that for all $t\geqslant 0$,
    \[
        \|\nabla \xi (t)\|_{L^{\infty}} \geqslant C\exp \left(C e^{Ct}\right)\,.
    \]
    \item If $U=\mathbb{T}^2$, then there exists $t_0>0$ such that for $t\geqslant t_0$ there holds: 
    \[
        \sup_{t'\leqslant t} \|\nabla \xi (t')\|_{L^{\infty}} \geqslant e^{t}\,.
    \]
\end{enumerate}
\end{theorem}

The next question is then to quantify how likely it is for an initial data to produce such growth. Our result in this direction is the following. 

\begin{theorem}[Growth estimates]\label{4sideTheo2d} Let $s>2$ and $\mu_s$ being the associate invariant measure for~\eqref{4Euler2} constructed by Theorem~\ref{4mainTheo2d3d} on $H^{s}(\mathbb{T}^2,\mathbb{R}^2)$. Then we have the following estimates.  
\begin{enumerate}[label=(\textit{\roman*})] 
    \item For $\mu_s$-almost every $u_0\in H^s$ the associate unique global solution $u \in \mathcal{C}^0(\mathbb{R}_+,H^s(\mathbb{T}^2,\mathbb{R}^2))$ from Theorem~\ref{4mainGlobalWolib} obeys the following growth estimate:
    \begin{equation}
        \label{4eqPolGrowth}
        \|u(t)\|_{H^{\sigma}} \leqslant C(u_0, s, \sigma) t^{\alpha}\,,
    \end{equation}
    for all $t \geqslant 1$, $\sigma \in (1,s]$ and any $\alpha > \frac{\sigma -1}{2s-\sigma-1}$.
    \item If $s>3$ then we have the following corollary:
    \[
        \|\nabla \xi (t)\|_{L^{\infty}} \leqslant C(u_0, s, \sigma) t^{\alpha}\,,
    \]
    for all $t \geqslant 1$, $\sigma \in (1,s]$ and any $\alpha > \frac{\sigma -1}{2s-\sigma-1}$.
\end{enumerate}
\end{theorem}

\subsubsection{Main limitation of the results}

Let us start with two comments regarding Theorem~\ref{4sideTheo2d}: 

\begin{itemize}
    \item It is interesting to see that by letting $s$ grow in Theorem~\ref{4sideTheo2d} we can obtain arbitrarily slow polynomial growth bounds. For example, let $\varepsilon >0$ and $\sigma = 2$, and also $s>2$ large enough so that $\frac{1}{2s-3}=\frac{\sigma - 1}{2s-\sigma-1} < \varepsilon$. Then we see that $\mu_s$ almost every $u_0 \in H^s$ satisfies $\|u(t)\|_{H^2}\leqslant C(u_0, s) t^{\varepsilon}$. An interesting question would be to investigate the limiting case $s \to \infty$, and see whether one can construct a measure $\mu_{\infty}$, supported on $\bigcap_{s>0} H^s$, invariant for~\eqref{4Euler2}, and such that on the support of this measure, almost every initial data gives rise to a bounded solution (globally in time).  
    \item The fact that $\mathbb{E}[\|u(t)\|_{H^s}^2]=C(s) \in (0,\infty)$ is \textit{a priori} not preventing the existence growing solutions to~\eqref{4Euler2}. In fact, let us start by recalling that by Theorem~\ref{4mainTheo2d3d}~(\textit{iv}) we know that for all $R>0$ there holds $\mu_s(u \in H^s, \|u\|_{H^s}>0)$. Under a stronger assumption (which we do not claim to hold in our case), for example that the sets $\{\|u(n)\|_{H^s}>R\}_{n \geqslant 0}$ are pairwise independent, the Borel-Cantelli lemma would imply that $\limsup_{n \to \infty} \|u(n)\|_{H^s} \geqslant R$ for $\mu_s$ almost every initial data. Therefore, taking the countable intersection on $R \in \mathbb{N}$ one obtains that $\limsup_{t \to \infty} \|u(t)\|_{H^s} = \infty$ for $\mu_s$ almost every initial data. 
    
    Note that some versions of the Borel-Cantelli require weaker assumptions than the pairwise independence of the sets $\{\|u(n)\|_{H^s}>R\}_{n \geqslant 0}$, such as a precise control of $\mu(\{\|u(n)\|_{H^s}>R\} \cap \{\|u(m)\|_{H^s}>R\})$, but estimating these quantities would require a better understanding of the measure $\mu_s$.      
\end{itemize}

Let $\mu_s$ be a measure constructed by Theorem~\ref{4mainTheo2d3d} in dimension $2$. The main limitation of our result is the following: if $u\in H^{s}$ is a stationary solution, then Theorem~\ref{4mainCoro2d} implies that $\mu_s(\{u\})=0$, but this does not prevent the measure $\mu$ to be supported \textit{exclusively} on stationary solutions. 

In the case that the measure $\mu_s$ does not concentrate on stationary solutions, then it is possible that growing solutions exist (in $H^{\sigma}$ norms), and Theorem~\ref{4sideTheo2d} shows that $\mu_s$ almost surely the growth rate is at most polynomial. 

In the eventuality of $\mu_s$ concentrating on stationary solutions, this would exhibit a stability property of the set of stationary solutions with respect to the approximation procedure and shows that our compactness procedure does not extract measures with non-trivial dynamical properties.

This question seems both fundamental and non-trivial. We highlight that this has been raised in several works~\cite{bedrossianCotiZelati,sverakVicol} and also very recently in~\cite{flodesSy}. To the knowledge of the author, there are no available works which are able to rule out the stationary solution supported case. For example this is a limitation in~\cite{flodesSy}. Finally, let us mention that in finite dimension there exists a positive result obtained in~\cite{mattinglyPardoux}, where a finite dimensional Euler system is considered.

\subsection{Existing results pertaining to the construction of invariant measures for the Euler equations} 

Theorems~\ref{4mainTheo2d3d} asserts the existence of a measure $\mu$, invariant under the dynamics of~\eqref{4Euler2}. We recall the main methods which produce invariant measures for partial differential equations. To the knowledge of the author there exist at least three such techniques. 

\begin{enumerate}[label=(\textit{\roman*})]
    \item The Gibbs-measure invariant technique, introduced by Bourgain in~\cite{bourgain, bourgain2d} and many authors after him. This technique is adapted for Hamiltonian PDE's. In the context of the two-dimensional Euler equations some results have been obtained by Albeverio-Cruzeiro~\cite{albeverioCruzeiro} and Flandoli~\cite{flandoli}. 
    \item Propagation of Gaussian initial data, initiated by Burq-Tzvetkov in~\cite{burqTzvetkov2, burqTzvetkov} in the context of wave equations. It consists of solving the equation with initial data taking the form $u_0=\sum_{n \in \mathbb{Z}} g_nu_ne_n(x)$ where $(g_n)_{n\in \mathbb{Z}}$ are identically distributed independent Gaussian random variables. 
    \item The \textit{fluctuation-dissipation method} of Kuksin. It consists in approximating the considered equation with a dissipation term and a fluctuating (random) term and in constructing invariant measures for these approximations. Then the basic idea is to take the vanishing viscosity limit, and retain some properties of the measures. Kuksin obtained the following result:
    % The question of producing an invariant measure for~\eqref{4Euler2} in dimension $2$ has been studied by Kuksin in~\cite{Kuksin2004}. More precisely it is proven the following. 
\end{enumerate}
\begin{theorem}[Kuksin, \cite{Kuksin2004, kuksin}] There exists a topological space $X \subset H^2$ and a measure $\mu$ supported on $X$  such that $\mu$ is invariant under the dynamics of~\eqref{4Euler2}. Moreover the measure $\mu$ is such that:
\begin{enumerate}[label=(\textit{\roman*})]
    \item There exists a continuous function $p$, increasing and such that $p(0)=0$ satisfying the following: for any Borelian $A \subset \mathbb{R}$ there holds,
    \[
        \mu(\|u\|_{L^2} \in \Gamma) + \mu(\|\nabla u\|_{L^2} \in \Gamma) \leqslant p(|\Gamma|)\,.
    \]
    \item For any $A \subset H^2$ of finite Hausdorff dimension, $\mu(A)=0$. 
\end{enumerate}
\end{theorem}

Let us make a few comments. As we want to construct measures in higher regularity spaces, the method (\textit{i}) does not seem well adapted. Indeed, an invariant measure has been constructed in the space $H^{-1}$ by Flandoli in~\cite{flandoli}, which is too low in regularity for our purposes. The method (\textit{ii}) seems to be difficult to apply in our situation. However this method has the advantage to produce a measure whose support is dense in Sobolev spaces, and a very precise description of the measures. Finally, the method (\textit{iii}) appears to be more flexible than the others in the context of fluid mechanics or even dispersive PDEs. We refer to~\cite{shirikyanKuksin} for applications of this method to dispersive equations and the work of Sy~\cite{sy, sy2, sy3} and Sy-Yu~\cite{syYu1, syYu2}. The draw-back of the \textit{fluctuation-dissipation} method is that since the invariant measures are constructed by a compactness technique, the nature of the invariant measures is not as good as the Gaussian measures of method (\textit{ii}) or (\textit{i}). Nevertheless, by a suitable analysis one can often prove some good features of these measures. We refer to~\cite{kuksin} for more details.   

\subsection{Structure of the proof of the main results}

The proof essentially contains two main ingredients: the construction of invariant measures for~\eqref{4Euler2} at regularity $H^{s}$ following the original argument of Kuksin in~\cite{Kuksin2004}, and a globalisation argument of Bourgain in~\cite{bourgain}. This combination of techniques has already appeared in the work of Sy, see~\cite{sy} for example. 

We start by explaining how a global invariant measure with good properties is used to globalise the solutions controling their growth.

\subsubsection{The globalisation argument}

Let us recall an argument contained in~\cite{bourgain} which aims at extending a local Cauchy theory to a global theory with additional bounds on the growth of the Sobolev norms. Note that for our purposes, a deterministic local theory (even global in dimension $2$) is already known and we only need to estimate the growth of the solutions.

Using a Borel-Cantelli argument, we reduce the almost-sure growth estimate of Theorem~\ref{4sideTheo2d} to the following estimate: for any $\varepsilon>0$ and any $T>0$ there is a set $G_{\varepsilon,T} \subset H^s$ such that $\mu(G_{\varepsilon ,T}) \geqslant 1-\varepsilon$ and for $u \in G_{\varepsilon ,T}$ a solution $u$ exists on $[0,T]$ such that there holds 
\[\|u(t)\|_{H^s} \leqslant CT^{\alpha} \text{ for any } t\leqslant T\,.\]
In order to prove such a bound, the key ingredients are:
\begin{enumerate}[label=(\textit{\roman*})]
    \item the existence of a \textit{formal} invariant measure $\mu$ for the considered equation, enjoying nice decay estimates, for example subgaussian estimates of the form $\mu (\|u\|_{H^s}>\lambda) \leqslant C e^{-\lambda ^2}$ or weaker decay estimates; meaning that initial data are not likely to be of large norm. As the measure is invariant this implies that at any time, the solution is not likely to be of large norm. 
    \item a \textit{nice} local well-posedness theory.  
\end{enumerate}
Picking $0=t_0< \cdots < t_N=T$, we can ensure that $\|u(t_k)\|_{H^s}$ is small for any $k$ thanks to (\textit{i}). We control the growth of $\|u(t)\|_{H^s}$ for $t \in [t_k,t_{k+1}]$ thanks to the local well-posedness theory (\textit{ii}). We refer to Section~\ref{4sec5} for details. 

\subsubsection{Producing an invariant measure}

We explain the general strategy for producing invariant measures for~\eqref{4Euler2} at regularity $H^s$, following~\cite{Kuksin2004}. In the following we write $s=2+\delta$ where $\delta >0$. The case $\delta =0$ is precisely the content of~\cite{Kuksin2004}. We rewrite~\eqref{4Euler2} taking the Leray projection (which we denote by $\mathbf{P}$, see Section~\ref{4sec2} for a definition). The Euler equation~\eqref{4Euler2} now takes the form:
\begin{equation}
    \partial_t u + B(u,u) =0\,,
\end{equation}
where $B(u,u)\coloneqq\mathbf{P}(u\cdot \nabla u)$. 

We introduce a random forcing $\eta$ (see~\eqref{4refWhite} for a precise definition) and a dissipative operator $L\coloneqq L_{\delta} \coloneqq (-\Delta)^{1+\delta}$. Then, the Euler equation~\eqref{4Euler2} is approximated by a randomly forced hyper-viscous equation:
\begin{equation}
    \label{4NSvv2}
    \partial _t u_{\nu} + \nu Lu_{\nu} +B(u_{\nu},u_{\nu}) = \sqrt{\nu} \eta  \text{ and } \nabla \cdot u_{\nu} =0\,,
\end{equation}
with initial condition $u_{\nu}(0)=u_0\in H^{2+\delta}(\mathbb{T}^2,\mathbb{R}^2)$. 

With only slight modification of the argument the method in~\cite{kuksin}, Chapter~2 we will construct invariant measures $\mu_{\nu}$ to~\eqref{4NSvv2} concentrated on $H^{2+\delta}$. Then we will prove compactness of the family $(\mu_{\nu})_{\nu >0}$ in order to obtain an invariant measure $\mu$ for~\eqref{4Euler2}.

\subsubsection{Organisation of the paper} Section~\ref{4sec2} recalls some basic results that will be used in Section~\ref{4sec3} to construct global solutions to approximate equations. In Section~\ref{4sec4}, invariant measures are constructed and in Section~\ref{4sec5} and the proof of the main theorems are given. Section~\ref{4sec6} is devoted to the proof of Theorem~\ref{4sideTheo2d}. Some important results and computations are postponed to the Appendix for convenience.  

\subsection*{Acknowledgements} I want to warmly thank my advisors Nicolas Burq and Isabelle Gallagher for encouraging me, suggesting the problem and subsequent discussions. I also thank Sergei Kuksin for interesting discussions, suggesting~\cite{kuksin} and many comments on a preliminary version of this article. I also thank the anonymous referees for their remarks, greatly improving this article. 

\section{Notation and preliminary results}\label{4sec2}

\subsection{Notation}

We use the notation $\mathbb{Z}_0^2=\mathbb{Z} \times \mathbb{Z} \setminus\{(0,0)\}$. 

We write $A \lesssim_c B$ when there is a constant $C(c)$ depending on $c$ such that $A \leqslant C(c)B$. 

For a function $F$ we write $\mathrm{d}F(u;v)$ the differential of $F$ at $u$ evaluated at $v$ and $\mathrm{d}^2(u ; \cdot, \cdot)$ for the second order differential of $F$ at $u$. 

\subsubsection{Fuctional spaces} $H^s$ stands for the usual Sobolev spaces. In this article we use some variants of theses spaces. We still denote the complete space 
\[
    \left\{u \in H^s(\mathbb{T}^2,\mathbb{R}^2)\,, \text{ such that } \int_{\mathbb{T}^2}u=0\right\}\,,
\] 
by $H^s$ for convenience, and refer to it as the space of $H^s$ functions with zero-mean. We endow this space with the norms \[\|u\|_{\dot H^s}^2 = \sum_{n \in \mathbb{Z}^2} |n|^{2s} |\hat u (n)|^2\text{ and } \|u\|_{H^s}^2 = \sum_{n\in \mathbb{Z}^2} \langle n \rangle ^{2s} |\hat u (n)|^2\,.\]
Recall that these two norms are equivalent on the space of zero mean $H^s$ functions: 
\[
    \|\cdot \|_{\dot H^s} \leqslant \|\cdot \|_{H^s} \leqslant 2\|\cdot \|_{\dot H^s}\,.
\] 

Given a space $X$, the space $X_{\textbf{div}}$ refers to the space of functions $u \in X$ such that $\nabla \cdot u=0$, endowed with the norm of $X$. 

We introduce the Leray projector $\mathbf{P} : L^2(\mathbb{T}^2,\mathbb{R}^2) \to L^2_{\operatorname{div}}(\mathbb{T}^2,\mathbb{R}^2)$ defined as the Fourier multiplier with coefficients $M_{ij}(n)=\delta_{ij}-\frac{n_in_j}{|n|^2}$, where $n=(n_i)_{1\leqslant i \leqslant 2} \in \mathbb{Z}^2$. We define $\mathbf{P}_{\leqslant N}$ to be the projection on frequencies $|n|\leqslant  N$. We recall that with this definition $\mathbf{P}_{\leqslant N}$ is not continuous on $L^p$ when $p \neq 2$, but we will not use estimates in $L^p$.  

We set $B(u,v)\coloneqq \mathbf{P}(u\cdot \nabla v)$ defined for sufficiently smooth $u, v$. $B$ is then extended to $L^2_{\textbf{div}} \times L^{2}_{\textbf{div}}$ by duality by $\langle B(u,v), \varphi \rangle := -\langle u \otimes v : \nabla \varphi \rangle$ for such that $\nabla \varphi \in L^{\infty}$. 

Let $(X,\|\cdot\|_X)$ be a normed space and $I$ an interval. The Sobolev space $W^{s,p}(I,X)$ is endowed with the norm
\[\|u\|_{W^{s,p}(I,X)}^p \coloneqq \|u\|_{L^{p}(I,X)}^p + \iint_{I \times I} \frac{\|u(t)-u(t')\|_X^p}{|t-t'|^{1+s p}}\,\mathrm{d}t\,\mathrm{d}t'\,.\]

We also sometime use $L^p_TX$ as a shorthand for $L^p((0,T),X)$. 

\subsubsection{Probability theoretic notation} If $E$ is topological space then $\mathcal{P}(E)$ stands for the set of probability measures on $E$ and $\mathcal{B}(E)$ stands for the set of its Borelians. 

We consider a probability space $(\Omega, \mathcal{F}, \mathbb{P})$ and further assume $\mathcal{F}$ to be complete, that is $\mathcal{F}$ contains all the sets contained in zero probability measure measurable sets. Let $(\mathcal{F}_t)_{t \geqslant 0}$ be a filtration of $\mathcal{F}$, that is, $t \mapsto \mathcal{F}_t$ is non-increasing. The quadruple $(\Omega, \mathcal{F}, F_t, \mathbb{P})$ is called a filtered probability space.  

A process $(x(t))_{t\geqslant 0}$ is said to be \textit{progressively measurable} with respect to a filtration $(\mathcal{F}_t)_{t\geqslant 0}$ if for any $T>0$, $x_{\vert [0,T]}$ is $\mathcal{B}([0,T]) \otimes \mathcal{F}_T$ measurable. 

The expectation with respect to the probability $\mathbb{P}$ will be denoted $\mathbb{E}$ and when the expectation will be taken with respect to a measure $\mu$ we will write $\mathbb{E}_\mu$. 

\subsubsection{The noise}

We let $(\mathcal{G}_t)_{t\geqslant 0}$ be the completed filtration associated to identically distributed independent Brownian motions $(\beta_n(t))_{n\in\mathbb{Z}^2}$.

In this text we will use a Hilbert basis $(e_n)_{n \in \mathbb{Z}^2_0}$ of the space of $\{u\in L^2_{\operatorname{div}}(\mathbb{T}^2,\mathbb{R}^2), \int_{\mathbb{T}^2} u=0 \}$ given by 
\[
    e_n(x)\coloneqq \frac{(-n_2,n_1)^T}{\sqrt{2}\pi|n|} \left\{
    \begin{array}{cc}
        \sin (n\cdot x) & \text{ if } n_1>0 \text{ or } (n_1=0 \text{ and } n_2>0)  \\
        \cos(n\cdot x) & \text{ otherwise,}
    \end{array}
    \right.
\]
and for $|n|=\sqrt{n_1^2+n_2^2} \neq 0$. 

Then for any $n \in \mathbb{Z}^2_0$ we have $(-\Delta)e_n=|n|^2e_n$. 

We set $\eta (t) = \frac{\mathrm{d}}{\mathrm{d}t} \zeta (t)$ where 
\begin{equation}
    \label{4refWhite} 
    \zeta (t) = \sum_{n \in \mathbb{Z}_0^2}\phi_n \beta_n(t)e_n\,,
\end{equation}
for some numbers $(\phi_n)_{n\in\mathbb{Z}_0^2}$ to be chosen later and introduce 
$\mathcal{B}_k \coloneqq \sum_{n \in \mathbb{Z}^2_0}|n|^{2k} |\phi_n|^2$.

\subsection{Deterministic preliminaries}

In the following we gather some standard results that we will use on many occasions. We start with some basic estimates of the bilinear form in dimension $d=2$ which can be found in~\cite{kuksin} Lemma~2.1.6 and Proposition~2.1.7 for instance. 

\begin{lemma}[Properties of the bilinear form in dimension $2$]\label{4lemmaBilinear} Let $u\in H^1_{\operatorname{div}}(\mathbb{T}^2,\mathbb{R}^2)$ and $v,w \in H^1(\mathbb{T}^2)$. Then we have the following. 
\begin{enumerate}[label=(\textit{\roman*})]
    \item $\langle B(u,v),v\rangle _{L^2}=0$.
    \item $\langle B(u,u),u\rangle _{H^{1}}=0$.
    \item $\|B(u,v)\|_{H^{-1}} \lesssim \|u\|_{H^{\frac{1}{2}}}\|v\|_{H^{\frac{1}{2}}}$. 
\end{enumerate}
\end{lemma}

\begin{remark} The property (\textit{ii}) will appear to be a crucial algebraic cancellation which will be heavily used through the rest of the paper, and is very specific to the dimension $2$.
\end{remark}

We will also need a basic heat kernel estimate. 

\begin{lemma}\label{4lemmaHeat} Consider the operator $L=(-\Delta)^{1+\delta}$ for $\delta >0$. Let $s_0 \in \mathbb{R}$, $\alpha \geqslant 0$ and $F \in H^{s_0}(\mathbb{T}^2)$, with zero mean. For any $s< s_0+\alpha (1+\delta)$ and $t>0$ there holds
\begin{equation}
    \left\|e^{-tL}F\right\|_{H^{s}} \lesssim t^{-\frac{\alpha}{2}}\|F\|_{H^{s_0}}\,,
\end{equation}
where the implicit constant only depends on $\alpha$. 
\end{lemma}

\begin{proof} We take the Fourier transform and use the mean zero condition to write
\begin{align*}
    \left\|e^{-tL}F\right\|_{H^{s}}^2 &= \sum_{n \in \mathbb{Z}^2}\langle n \rangle ^{2s} e^{-2t|n|^{2(1+\delta)}}|\hat{F}(n)|^2 \\
    & = \sum_{n \in \mathbb{Z}_0^2} \langle n \rangle ^{2s}\frac{e^{-2t|n|^{2(1+\delta)}}t^{\alpha} |n|^{2\alpha (1+\delta)}}{t^{\alpha} |n|^{2\alpha (1+\delta)}}|\hat{F}(n)|^2\,.
\end{align*}
Then observe that the function $x \mapsto x^{\alpha}e^{-x}$ is bounded so that 
\[
    \left\|e^{-tL}F\right\|_{H^{s}}^2 \leqslant Ct^{-\alpha} \sum _{n \in \mathbb{Z}_0^2} \langle n \rangle ^{2(s-s_0)-2\alpha(1+\delta)} \langle n \rangle ^{2s_0}|\hat{F}(n)|^2 \,,
\]
which is bounded by $t^{-\alpha}\|F\|_{H^{s_0}}^2$ as soon as $s<s_0+\alpha (1+\delta)$. 
\end{proof}

\subsection{Probabilistic preliminaries}

We will use the Itô isometry in the following form:

\begin{theorem}[Itô isometry, \cite{DaPratoZabysck}] Let $F(t)$ be an $\mathcal{G}_t$-adapted process. Then, for any $t>0$ there holds
\begin{equation}
    \label{4eqIto}
    \mathbb{E}\left[ \left(\int_0^t F(t')\,\mathrm{d}\beta(t')\right)^2\right] = \mathbb{E}\left[\|F\|_{L^2(0,t)}^2\right]\,.
\end{equation}
\end{theorem}

Moreover we will often use the following facts: 
\begin{enumerate}[label=(\textit{\roman*})]
    \item When $F$ is a deterministic function, then $\displaystyle \int_0^tF(t')\,\mathrm{d}\beta_n(t')$ is a Gaussian random variable of mean $0$ and variance $\|F\|_{L^2((0,t)}^2$, see~\cite{DaPratoZabysck}. 
    \item If $Y$ is a Gaussian random variable on a Hilbert space $H$ (that is, $(Y,a)_H$ is a real Gaussian for any $a \in H$), then one has 
        \begin{equation}
            \label{4eqGaussBound}
            \|Y\|_{L^p_{\omega}} \lesssim \sqrt{p} \|Y\|_{L^2_{\omega}}\,,
        \end{equation}
    for any $p\geqslant 1$. 
\end{enumerate}

We also use the following well-known regularity criterion. 

\begin{lemma}[Kolmogorov,\cite{DaPratoZabysck}]\label{4kolmorogov} Let $(X(t))_{t\geqslant 0}$ be a stochastic process with values in a Banach space endowed with a norm $\| \cdot \|$. Assume that there exist $p \geqslant 1$ and $\alpha >0$ such that \[\mathbb{E}[\|X(t)-X(s)\|^p] \lesssim |t-s|^{1+\alpha}\,.\] Then for any $\varepsilon >0$ small enough, then $(X(t))_{t\geqslant 0}$ admits a modification (\textit{i.e.}, there is $(\tilde{X}(t))_{t\geqslant 0}$ such that for all $t \geqslant 0$, $\tilde X (t)=X(t)$ almost surely) that is almost-surely $(\frac{\alpha}{p}-\varepsilon)$ Hölder continuous. 
\end{lemma}

\section{Constructions of solutions to approximate equations}\label{4sec3}

In this section, we study the well-posedness of the following hyper-viscous equation:
\begin{equation}
    \tag{HVE${}_{\nu}^2$}
    \label{4eqNSv2}
    \left\{
    \begin{array}{ccc}
        \partial _t u_{\nu} + \nu Lu_{\nu} +B(u_{\nu},u_{\nu}) & = & \sqrt{\nu} \eta \\
       u_{\nu}(0) & = & u_0 \in H^1_{\operatorname{div}}(\mathbb{T}^2)\,,
    \end{array}
    \right.
\end{equation}
where we recall that $L=(-\Delta)^{1+\delta}$ for some $\delta >0$. $\eta = \frac{\mathrm{d}}{\mathrm{d}t} \zeta(t)$ where $\zeta(t)$ is defined by~\eqref{4refWhite} and satisfies $\mathcal{B}_{3+\delta}<\infty$.  

We let $(\Omega, \mathbb{P}, \mathcal{F})$ be a probability space with the Brownian motion completed filtration $(\mathcal{G}_t)_{t \geqslant 0}$. 

\begin{proposition}[Solutions to~\eqref{4eqNSv2}]
\label{4propNSsol}
Let $\nu >0$ and $u_0$ a $\mathcal{G}_0$-measurable random variable such that $u_0 \in H^1_{\operatorname{div}}$ almost surely. We also assume that the noise~$\eta$ is such that $\mathcal{B}_{3+2\delta}<\infty$. Then there exists a set $\Omega_1$ such that $\mathbb{P}(\Omega _1)=1$ and such that for any $\omega \in \Omega_1$:
\begin{enumerate}[label=(\textit{\roman*})]
    \item \eqref{4eqNSv2} has a unique global solution $t \mapsto u_{\nu}^{\omega}(t)$ associated to $u_0^{\omega}$, in the space $\mathcal{C}^0(\mathbb{R}_+,H^1_{\operatorname{div}})$.
    \item Moreover $u_{\nu}^{\omega} \in L^2_{\operatorname{loc}}(\mathbb{R}_+,H^{2+\delta})$.
    \item $u_{\nu}$ may be written in the form 
    \[u_{\nu}(t)=u_0+\int_0^t f(s)\,\mathrm{d}s + \sqrt{\nu}\zeta (t)\,,\]
    where equality holds in $H^{-\delta}$ and where $f(s)$ is a $\mathcal{G}_s$-progressively measurable process satisfying that almost surely $f \in L^2_{\operatorname{loc}}(\mathbb{R}_+, H^{-\delta})$. In particular $u(t)$ is a $\mathcal{F}_t$-progressively measurable process.
    \item There exists a measurable map $U : H^{1} \times \mathcal{C}^0(\mathbb{R}_+,H^1) \to \mathcal{C}^0(\mathbb{R}_+,H^1)$, continuous in its first variable, such that for any $\omega \in \Omega_1$, $u^{\omega}_{\nu}=U(u_0,\zeta)$.
\end{enumerate}
\end{proposition}

\begin{remark} The reader already aware of the methods used in~\cite{Kuksin2004} can skip the proof provided in this section, as the proof follows the same lines. 
\end{remark}

A general strategy for solving a nonlinear stochastic partial differential equation is to seek for solutions which have a particular structure, designed to eliminate the randomness and apply a fixed-point argument. 

We explain the general scheme for proving Proposition~\ref{4propNSsol}. In order to solve~\eqref{4eqNSv2} we first look for solutions $z_{\nu}$ to the following equation:
\begin{equation}
    \label{4eqSto}
    \left\{
    \begin{array}{ccc}
         \partial_t z_{\nu}(t) + \nu Lz_{\nu} (t) &=& \sqrt{\nu} \eta (t)  \\
         z_{\nu}(0) &=& 0\,.
    \end{array}
    \right.
\end{equation}
Then we seek for solutions to~\eqref{4eqNSv2} taking the form $u_{\nu}=z_{\nu}+v_{\nu}$ where $v_{\nu}$ formally satisfies:
\begin{equation}
    \label{4eqNSv}
    \left\{
    \begin{array}{ccc}
         \partial_t v_{\nu} + \nu Lv_{\nu} +B(z_{\nu}+v_{\nu},z_{\nu}+v_{\nu})&=&0\\
         v_{\nu}(0)&=&u_0\,, 
    \end{array}
    \right.
\end{equation}
This is a deterministic equation and can be solved in the space $\mathcal{C}^0(\mathbb{R}_+,H^1_{\operatorname{div}})$. 

A solution to~\eqref{4eqSto} is explicitly given by
\begin{equation}
    \label{4eqZnu}
    z_{\nu}(t)=\sqrt{\nu}\sum_{n \in \mathbb{Z}_0^2} \phi_n \left(\int_0^t e^{-\nu (t-t')L}\mathrm{d}\beta _n (t')\right)e_n \in H^1_{\operatorname{div}}(\mathbb{T}^2,\mathbb{R}^2)\,,
\end{equation}

as this can be checked by the Itô formula for example, see Appendix~\ref{4appA}. Moreover, the process $z_{\nu}$ satisfies the following properties. 

\begin{lemma}\label{4lemmaSto} Under the assumptions of Proposition~\ref{4propNSsol}, there exists a set $\Omega_1$ such that $\mathbb{P}(\Omega_1)=1$ and which satisfies the following properties for any $\omega \in \Omega_1$. Let us denote by $z$ the solution $z_{\nu}$ in~\eqref{4eqZnu}. Then:
\begin{enumerate}[label=(\roman*)]
    \item $z \in \mathcal{C}^0(\mathbb{R}_+,H_{\operatorname{div}}^1) \cap L^2_{\operatorname{loc}}(\mathbb{R}_+,H^{2+\delta})$. 
    \item There exists $\alpha >0$ such that $z$ almost surely belongs to the space $\mathcal{C}^{\alpha}_{\operatorname{loc}}(\mathbb{R}_+,W^{2,4})$.
    \item $\mathbb{E}[\|z(t)\|^2_{\dot H^1}] \leqslant \frac{\mathcal{B}_1}{2}\nu t$, for any $t\geqslant 0$.  
\end{enumerate}
\end{lemma}

This lemma can be proven using the techniques in~\cite{kuksin}, Chapter~2; Section~4, as well as the following lemma, allowing to solve~\eqref{4eqNSv} globally. 

\begin{lemma}\label{4propLocalGlobal} Let $u_0$ a random variable such that $u_0 \in H^1_{\operatorname{div}}$ almost surely. Then there exists a set $\Omega_2$ of probability $1$ such that for any $\omega \in \Omega_2$, the associate Cauchy problem to~\eqref{4eqNSv} is globally well-posed in $\mathcal{C}(\mathbb{R}_+,H^1_{\operatorname{div}}(\mathbb{T}^2))$. Furthermore:
\begin{enumerate}[label=(\textit{\roman*})]
    \item For any $T>0$, the flow map $H^1 \to \mathcal{C}^0([0,T],H^1_{\operatorname{div}})$  defined by $u_0 \mapsto v_{\nu}$ is locally Lipschitz.
    \item $v_{\nu} \in L^2_{\operatorname{loc}}(\mathbb{R}_+,H^{2+\delta})$. 
\end{enumerate}
\end{lemma}

\begin{proof}[Proof of Proposition~\ref{4propNSsol}] Part (\textit{i}) follows from Lemma~\ref{4lemmaSto} and Lemma~\ref{4propLocalGlobal} and part (\textit{ii}) follows from Lemma~\ref{4propLocalGlobal}. 

We prove (\textit{iv}). In fact, Lemma~\ref{4propLocalGlobal} proves that $v_{\nu}$ is a continuous function of $z_{\nu}$ and $u_0$ as a corollary of the local well-posedness theory, furthermore, there exists a measurable function $Z : \mathcal{C}^0(\mathbb{R}_+,H^1) \to \mathcal{C}^0(\mathbb{R}_+,H^1)$ such that $z_{\nu}=Z(\zeta)$, where $Z(\zeta)$ is given by~\eqref{4eqZnu}. The measurability comes from the fact that the maps $Z_M$ defined by 
\[
    Z_M(\zeta)(t)=\sqrt{\nu}\sum_{|n|\leqslant M} \phi_n \left(\int_0^te^{-\nu (t-t')L}\,\mathrm{d}\beta_n(t')\right)e_n\,,
\]
are continuous, hence measurable; and that $Z=\displaystyle\lim_{M \to \infty} Z_M$ almost surely. For details, see~\cite{kuksin}, Remark~2.4.3. 

In order to prove (\textit{iii}) we write $u_{\nu}(t)=u_0+\int_0^t f(s)\,\mathrm{d}s + \sqrt{\nu}\zeta (t)$ where $f(s)=-\nu Lu_{\nu}(s)+B(u_{\nu}(s),u_{\nu}(s))$ and $\displaystyle\zeta (t):=\sum_{n \in \mathbb{Z}^2}\phi_n\beta_n(t)e_n$. Then $f(s)$ is $\mathcal{G}_s$ progressively measurable. Indeed, we know that $s \mapsto z_{\nu}(s)$ is progressively measurable by construction and properties of the stochastic integral, and we also know that $s \mapsto v_{\nu}(s)$ is progressively measurable thanks to (\textit{iv}). Finally we explain why $f\in L^2_{\operatorname{loc}} H^{-\delta}$. First, since $u_{\nu}$ lies in $L^2_{\operatorname{loc}} H^{2+\delta}$ thanks to Lemma~\ref{4propLocalGlobal} we obtain $Lu_{\nu} \in  L^2_{\operatorname{loc}} H^{-\delta}$. For the bilinear term we use Lemma~\ref{4lemmaBilinear} to bound 
\[
    \|B(u_{\nu},u_{\nu})\|_{H^{-\delta}} \leqslant \|B(u_{\nu},u_{\nu})\|_{L^2}  \leqslant \|u_{\nu}\|_{H^1}^2\,.
\]
Since $u_{\nu} \in \mathcal{C}^0(\mathbb{R},H^1_{\operatorname{div}})$, it follows that $B(u_{\nu},u_{\nu}) \in L^2_{\operatorname{loc}}(\mathbb{R}_+,H^{-\delta})$ and thus the result. 
\end{proof}

\section{Construction of invariant measures}\label{4sec4}

We study the invariance properties of the solutions constructed in Section~\ref{4sec3} by Proposition~\ref{4propNSsol}. In order to do so, we remark that the processes constructed by Proposition~\ref{4propNSsol} enjoy a Markovian structure. We explain how: enlarge the probability set defining $\tilde{\Omega}:=H^1_{\operatorname{div}} \times \Omega$ endowed with the $\sigma$-algebra $\mathcal{B}(H^1)\otimes \mathcal{F}$ and the filtration $\tilde{\mathcal{F}}_t:=\mathcal{B}(H^1)\otimes \mathcal{G}_t$ and denote $\tilde{\omega}=(v,\omega)$ for elements in $\tilde{\Omega}$. We let $\tilde{u}^{\tilde{\omega}}(t):=u_{\nu}^{\omega}(t,v)$, standing for the solution constructed by Proposition~\ref{4propNSsol} with initial data $u_0=v \in H^1$. Let $\mathbb{P}_v:=\delta_v \otimes \mathbb{P}$. Then $(u_{\nu}(t),\mathbb{P}_v)$ forms a Markovian system, as it satisfies the \textit{Markov property}:
\[\mathbb{P}_v(u_{\nu}(t+s) \in \Gamma \vert \mathcal{F}_s)=\mathbb{P}_{u_{\nu}(s)}(u_{\nu}(t)\in \Gamma)\,,\]
which comes from Proposition~\ref{4propNSsol}, (\textit{iv}). Let us remark that if $U_t$ denotes the restriction of the map $U$ of Proposition~\ref{4propNSsol}, then $\mathcal{L}(u_{\nu}(t))=(U_t)_{*}(\delta _v \otimes m_{\zeta, T})$, where $m_{\zeta , T} = \mathcal{L}(\zeta _{\vert [0,T]})$. In particular we have $\mathcal{L}(u_0)=(U_0)_*\mathcal{L}(u_{\nu})$. 

\subsection{Existence of an invariant measure}

In order to construct an invariant measure for~\eqref{4eqNSv2} we follow the strategy in~\cite{kuksin} which consists in applying the Krylov-Bogolioubov argument.

Let us introduce some more notation, let us denote $P_t(u,\Gamma):=\mathbb{P}_v(u(t)\in \Gamma)$, and define 
the semi-group $\mathfrak{B}_t : L^{\infty} \to L^{\infty}$ defined, for any $f\in L^{\infty}$ by $\mathfrak{B}_t(f)= z \mapsto \int_{H^1} f(z) P_t(v,\mathrm{d}z)$. We define the dual $\mathfrak{B}_t^* : \mathcal{P}(H^1) \to \mathcal{P}(H^1)$ by $\mathfrak{B}_t^*(\mu):= \Gamma \mapsto \int_{H^1}P_t(v,\Gamma)\mu(\mathrm{d}v)$, and observe that $\mathcal{L}(u(t))=\mathfrak{B}_t^*(\mathcal{L}(u_0))$. 

We start with some higher Sobolev estimates for processes such that $\mathcal{L}(u_0)=\delta _0$.  

\begin{lemma}\label{4lemmaBound} Assume that $\mathcal{L}(u_0)=\delta_0$ and let $u_{\nu}$ being the corresponding process produced by Proposition~\ref{4propNSsol}. Then there exists a constant $C>0$ independent of $\nu$ such that for any $t>0$, \[\mathbb{E}\left[\int_0^t\|u_{\nu}(t')\|^2_{\dot{H}^{2+\delta}}\,\mathrm{d}t'\right] \leqslant Ct\,.\]
\end{lemma}

\begin{proof} We apply the Itô formula (see Proposition~\ref{4ito} and the subsequent discussion) to the functional $F(u):=\|u\|_{\dot{H}^1}^2$ and we find that for all $t \in \mathbb{R}$, 
\[
    \mathbb{E}\left[\|u_{\nu}(t)\|^2_{\dot{H}^1}\right]-\mathbb{E}\left[\|u_0\|^2_{\dot{H}^1}\right] + 2\nu\mathbb{E} \left[\int_0^t \|u_{\nu}(t')\|_{\dot{H}^{2+\delta}}^2\,\mathrm{d}t'\right] = \nu\mathcal{B}_1t\,.
\]
Then the result follows from $\mathbb{E}[\|u_0\|^2_{\dot{H}^1}]=0$, since $\mathcal{L}(u_0)=\delta_0$.  
\end{proof}

We can now state the main results of the section. 

\begin{proposition}\label{4propUnifBounds}The Markov system $(u_{\nu} (t),\mathbb{P}_v)_{v \in H^1}$ admits a stationary measure. Moreover, for any stationary measure $\mu_{\nu} \in \mathcal{P}(H^1)$, the following properties hold. 
\begin{enumerate}[label=(\textit{\roman*})]
    \item $\mathbb{E}_{\mu_{\nu}} \left[\|u\|^2_{\dot{H}^{1+\delta}}\right]=\frac{\mathcal{B}_0}{2}$. 
    \item $\mathbb{E}_{\mu_{\nu}} \left[\|u\|^2_{\dot{H}^{2+\delta}}\right] = \frac{\mathcal{B}_1}{2}$. 
    \item There exists $\gamma >0$, and $C<\infty$, only depending on the noise parameters $(\phi_n)_{n \in \mathbb{Z}_0^2}$ such that one has $\mathbb{E}_{\mu_{\nu}} \left[e^{\gamma\|u\|^2_{\dot H^1}}\right] \leqslant C$.
\end{enumerate}
\end{proposition}

\begin{remark} We recall that the measures $\mu_{\nu}$ also depend on $\delta$, as they depend on the regularising operator $L=L_{\delta}$. 
\end{remark}

\begin{proof} Let us denote by $u$ the solution starting at $u_0$ with law $\mathcal{L}(u_0)=\delta_0$, and let $\lambda _t := \mathfrak{B}_t^* \delta_0$. We introduce $\bar{\lambda_t} := \frac{1}{t} \displaystyle\int_0^t \lambda_{t'} \,\mathrm{d}t'$. In order to apply the Krylov-Bogolioubov theorem we need to show that the family $(\bar{\lambda_t})_{t>0}$ is tight in $H^1$. Since the embedding $H^{2+\delta} \hookrightarrow H^1$ is compact it is sufficient to prove that \[\sup_{t>0}\bar{\lambda} _t \left(H^1 \setminus B_{H^{2+\delta}}(0,R)\right) \underset{R \to \infty}{\longrightarrow} 0\,.\]
Observe that thanks to Lemma~\ref{4lemmaBound}, we have
\begin{align*}
\bar{\lambda} _t \left(H^1 \setminus B_{H^{2+\delta}}(0,R)\right) & \leqslant \frac{1}{t} \int_0^t \mathbb{P}(\|u_{\nu}(t')\|_{H^{2+\delta}}>R)\,\mathrm{d}t' \\
&\leqslant \frac{1}{tR^2} \int_0^t \mathbb{E}\left[\|u_{\nu}(t')\|^2_{H^{2+\delta}}\right] \\
& \leqslant \frac{C}{R^2}\,,
\end{align*}
which goes to zero uniformly in $t>0$. Then the Krylov-Bogolioubov theorem (see~\cite{dudley}) ensures the existence of stationary measures. Let $\mu_{\nu}$ be such a stationary measure and let us prove the required estimates.  

(\textit{i}) and (\textit{ii}) are proven using the same argument. Let us only prove (\textit{ii}). The Itô formula just as in the proof of Lemma~\ref{4lemmaBound}: we take $u_{\nu}(t)$ a process solving \eqref{4eqNSv2} with stationary measure $\mu_{\nu}$. We can write, for any $t\geqslant 0$, thanks that the Itô formula:
\[\mathbb{E}\left[\|u_{\nu}(t)\|^2_{\dot{H}^1}\right]-\mathbb{E}\left[\|u_0\|^2_{\dot{H}^1}\right] + 2\nu\mathbb{E} \left[\int_0^t \|u_{\nu}(t')\|_{\dot{H}^{2+\delta}}^2\,\mathrm{d}t'\right] = \nu\mathcal{B}_1 t\,,\]
for any $t>0$. Using invariance, we obtain \[\int_0^t\left(2\mathbb{E}\left[\|u_{\nu}(t)\|^2_{\dot{H}^{2+\delta}}\right]-\mathcal{B}_1\right)=0\] for all $t\geqslant 0$ so that $\mathbb{E}\left[\|u_{\nu}(t)\|^2_{\dot{H}^{2+\delta}}\right]=\frac{\mathcal{B}_1}{2}$, for almost any $t>0$. However by invariance again (and since $u \mapsto \|u\|^2_{H^{2+\delta}}$ is a Borelian function of $H^1$), we know that the quantity $\mathbb{E}[\|u_{\nu}(t)\|^2_{\dot H^{2+\delta}}]$ is time-invariant, thus finite for \textit{all} $t$, thus identically equal to $\mathcal{B}_1$. 

The same argument applied to $\|u_{\nu}(t)\|_{L^2}^2$ instead of $\|u_{\nu}(t)\|_{\dot H^1}^2$ applies in order to prove (\textit{i}), as the main observation being the cancellation $(B(u,u),u)_{L^2}=0$. 

(\textit{iii}) also comes from the Itô formula and we refer to Appendix~\ref{4appA} for more details of the computations. The Itô formula applied to the functional defined by $G(u)=e^{\gamma \|u\|^2_{\dot H^1}}$ yields that for any $\gamma >0$ and any $t>0$:
\begin{align*}
    \mathbb{E}\left[e^{\gamma \|u_{\nu}(t)\|_{\dot H^1}^2}\right]&=\mathbb{E}\left[e^{\gamma \|u_0\|_{\dot H^1}^2}\right] \\
    &+2\gamma \nu \mathbb{E}\left[\int_{0}^t e^{\gamma \|u_{\nu}(t')\|_{\dot H^1}^2}\left(\frac{\mathcal{B}_1}{2}-\|u_{\nu}(t')\|_{\dot H^{2+\delta}}^2 + \gamma \sum_{n\in\mathbb{Z}^2} |n|^2|\phi_n|^2 |u_n(t')|^2\right)\,\mathrm{d}t'\right]\,,
\end{align*}
where $u_n(t)=(u_{\nu}(t),e_n)_{L^2}$. 

Since $\max \{|\phi _n^2|, n\in \mathbb{Z}^2\}< \infty$, choosing $\gamma$ such that $\gamma \sup |\phi_n|^2 \simeq \frac{1}{2}$ leads to the inequality
\[\mathbb{E}\left[e^{\gamma \|u_{\nu}(t)\|_{\dot H^1}^2}\right]-\mathbb{E}\left[e^{\gamma \|u_0\|_{\dot H^1}^2}\right] \leqslant C\nu \mathbb{E}\left[\int_{0}^t A(t')\,\mathrm{d}t'\right]\,,\]
where $A(t')\coloneqq e^{\gamma \|u_{\nu}(t')\|_{\dot H^1}^2}\left(C-\|u_{\nu}(t')\|_{\dot H^{2+\delta}}^2\right)$ and $C>0$ does not depend on $\nu$. 

Remark that if $\|u_{\nu}(t)\|^2_{\dot H^{2+\delta}}>2C$ then 
\[ 
    e^{\gamma \|u_{\nu}(t')\|^2_{\dot H^1}}\left(2C-\|u_{\nu}(t')\|_{\dot H^{2+\delta}}^2\right) \leqslant 0\,,
\]
and if $\|u_{\nu}(t)\|^2_{\dot H^{2+\delta}}\leqslant 2C$ then
\[
    e^{\gamma \|u_{\nu}(t')\|_{\dot H^1}^2}\left(2C-\|u_{\nu}(t')\|_{\dot H^{2+\delta}}^2\right) \leqslant 2Ce^{2\gamma C}\,,
\]
so that in any case 
\[
e^{\gamma \|u_{\nu}(t')\|_{\dot H^1}^2}\left(2C-\|u_{\nu}(t')\|_{\dot H^{2+\delta}}^2\right) \leqslant \underbrace{2Ce^{2\gamma C}}_{C_1}\,,
\]
which implies $A(t') \leqslant C_1 - Ce^{\gamma \|u_{\nu}(t')\|_{\dot H^1}^2}$ and therefore introducing $Y(t)\coloneqq \mathbb{E}\left[e^{\gamma \|u_{\nu}(t)\|_{\dot H^1}^2}\right]$, we have obtained
\[
    Y(t)+C\nu\int_0^t Y(t')\,\mathrm{d}t' \leqslant Y(0)+C_1 \nu t\,,
\]
where $C$ and $C_1$ do not depend on $\nu$ and the Grönwall lemma gives
\begin{equation}
    \label{4eqSplitGronwal}
    Y(t) \leqslant e^{-C\nu t}Y(0)+C_2\,,
\end{equation}
where $C_2$ is a function of $C$ and $C_1$, and provided $Y(0)<\infty$. Then recalling that invariance implies $Y(t)=Y(0)$ and letting $t \to \infty$ this yields $Y(0) \leqslant C_2$, a finite constant independent of $\nu$. 

Let us prove that $Y(0)$ is a finite quantity. To this end we set approximations 
\[
    Y_{N}(t):=\mathbb{E}\left[f_N(u_{\nu}(t))\right] \text{ where } f_N(v) \coloneqq \exp\left(\gamma\min\{\|v\|_{\dot H^1}^2,N^2\}\right)\,.
\] 
Then we write:
\[Y_{N}(t)=\mathbb{E}\left[f_N(u_{\nu}(t))\mathbf{1}_{\|u_0\|_{\dot H^1}>R}\right] + \mathbb{E}\left[f_N(u_{\nu}(t))\mathbf{1}_{\|u_0\|_{\dot H^1}\leqslant R}\right]\,,\]
for $R \geqslant N$. The first term is bounded by $e^{\gamma N^2}\mu_{\nu}(\|u\|_{\dot H^1}>R)$ and the second one using the previous estimates. Using the Markov inequality, part (\textit{i}) of the proposition and~\eqref{4eqSplitGronwal} gives
\begin{align*}
    Y_N(0)=Y_N(t) & \leqslant e^{\gamma N^2}\mu_{\nu}(\|u\|_{\dot H^1}>R) + \mathbb{E}_{\mu _{\nu}}\left[e^{\gamma \|u\|_{\dot H^1}^2} \vert \|u\|_{\dot H^1} \leqslant R \right]\\
    & \leqslant e^{\gamma N^2}\mu_{\nu}(\|u\|_{\dot H^{2+\delta}}>R + \mathbb{E}_{\mu _{\nu}}\left[e^{\gamma \|u\|_{\dot H^1}^2} \vert \|u\|_{\dot H^1} \leqslant R \right]\\
    & \lesssim C(\nu)e^{\gamma N^2} R^{-2} + e^{-C\nu t+R^2}+C_2\,.
\end{align*}
First take the limit $t \to \infty$, then $R \to \infty$ to get $Y_N(0) \leqslant C_2$, which is uniform in $N$. By the monotone convergence theorem we deduce that $Y(0)< \infty$. 
\end{proof}

\subsection{Tightness and limit}

Once and for all we fix some stationary measures $\mu_{\nu}$. Our next task is to pass to the limit $\nu \to 0$. In order to do so, we need to prove several compactness estimates. We denote by $\bar{\mu}_{\nu}$ the law of $u_{\nu}(\cdot)$. We will use compactness arguments based on the Aubin-Lions-Simon criterion, which we recall, and refer to Corollary~9 in~\cite{simon} for a proof of (\textit{ii}).

\begin{theorem}[Aubin-Lions-Simon compactness theorem]\label{4aubinLions} Let $I$ be a compact interval. Let $B_0, B, B_1$ be three separable complete spaces. Assume that $B_0 \hookrightarrow B \hookrightarrow B_1$ continuously, the embedding $B_0 \hookrightarrow B$ being compact.
\begin{enumerate}[label=(\textit{\roman*})]
    \item Let $W := \{u \in L^{p}(I,B_0) \text{ and } \partial _t u \in L^{q}(I,B_1)\}$. Then for $p<\infty$, $W$ is compactly embedded into $L^p(I,B)$. If $p= \infty$ and $q>1$ then $W$ is compactly embedded into $\mathcal{C}^0(I,B)$. 
    \item Let $s_0, s_1 \in \mathbb{R}$ and $1 \leqslant r_0, r_1 \leqslant \infty$. Let $W:=W^{s_0,r_0}(I,B_0) \cap W^{s_1,r_1}(I,B_1)$. Assume that there exists $\theta \in (0,1)$ such that for all $v \in B_0 \cap B_1$ one has \[\|v\|_{B} \lesssim \|v\|_{B_0}^{1-\theta}\|v\|_{B_1}^{\theta}\,.\] Let \[s_{\theta}:=(1-\theta)s_0+\theta s_1 \text{ and } \frac{1}{r_{\theta}}=\frac{1-\theta}{r_0}+\frac{\theta}{r_1}\,,\]
    then if $s_{\theta} > \frac{1}{r_{\theta}}$, $W$ is compactly embedded into $\mathcal{C}^0(I,B)$. 
\end{enumerate}
\end{theorem}

Proposition~\ref{4propUnifBounds} already asserts that $\bar{\mu}_{\nu}(L^2_{\operatorname{loc}}(\mathbb{R}_+,H^{2+\delta})\cap \mathcal{C}^0(\mathbb{R}_+,H^1))=1$. 

We will prove compactness estimates on $\bar{\mu}_{\nu}$ rather than simply on $\mu_{\nu}$, because not only do we want to make the measures $\mu_{\nu}$ converge to some measure $\mu$, but we want the processes $u_{\nu}$ to converge to a process $u$ solving the Euler equations. 

We want to prove that the family of measures $(\bar{\mu}_{\nu})_{\nu >0}$ is tight in
\[Y:=\mathcal{C}^0(\mathbb{R}_+,H^{1-\varepsilon}) \cap L^2_{\operatorname{loc}}(\mathbb{R}_+,H^{2+\delta - \varepsilon})\,,\] for a fixed $\varepsilon>0$ that we may take arbitrarily small. It is sufficient to prove the tightness on every time interval $I_n=[0,n]$, namely that for all $n \geqslant 1$, $(\bar\mu _{\nu})_{\nu >0}$ is tight in 
\[Y_n:=\mathcal{C}^0(I_n,H^{1-\varepsilon}) \cap L^2_{\operatorname{loc}}(I_n,H^{2+\delta -\varepsilon})\,.\]
Since the time interval does not play any specific role, we will prove tightness in
\[Y_1:=\mathcal{C}^0(I,H^{1-\varepsilon}) \cap L^2_{\operatorname{loc}}(I,H^{2+\delta-\varepsilon})\,,\]
where $I=[0,1]$. 

In order to do so, we introduce the space 
\begin{align*}
    X&:=L^2(I,H^{2+\delta}) \cap  \left(H^1(I,H^{-\delta}) + W^{\frac{3}{8},4}(I,H^{1-2\varepsilon})\right) \\
    &= \underbrace{L^2(I,H^{2+\delta}) \cap H^1(I,H^{-\delta})}_{X_1}  + \underbrace{L^2(I,H^{2+\delta}) \cap W^{\frac{3}{8},4}(I,H^{1-2\varepsilon})}_{X_2}\,.
\end{align*}

We will prove that following lemmata.

\begin{lemma}\label{4lemmaCompactness} Both spaces $X_1$ and $X_2$ are compactly embedded into $Y_1$. 
\end{lemma}

\begin{lemma}\label{4lemmaTight} The sequence $(u_{\nu})_{\nu >0}$ is bounded in $L^2(\Omega, X)$. 
\end{lemma}

These results then imply that the family $(\bar{\mu}_{\nu})_{\nu>0}$ is tight in $Y$.

\begin{proof}[Proof of Lemma~\ref{4lemmaCompactness}] Since $H^{2+\delta - \varepsilon} \hookrightarrow H^{2+\delta}$ compactly, we have compact embeddings $X_1 \hookrightarrow L^2(I,H^{2+\delta})$ and $X_2\hookrightarrow L^2(I,H^{2+\delta})$ thanks to Theorem~\ref{4aubinLions}, (\textit{i}).  

The compactness of the embedding $X_1 \hookrightarrow \mathcal{C}^0(I,H^{1-\varepsilon})$ follows from Theorem~\ref{4aubinLions}, (\textit{ii}). Indeed, take $s_0=0$, $s_1=1$, $r_0=r_1=2$ and observe that $1-\varepsilon = (1-\theta)(2+\delta - \varepsilon)+\theta (-\delta)$ with $\theta > \frac{1}{2}$, so that $s_{\theta}>\frac{1}{r_{\theta}}$. 

The embedding  $X_2 \hookrightarrow \mathcal{C}^0(I,H^{1-\varepsilon})$ is compact thanks to Theorem~\ref{4aubinLions}, (\textit{ii}). This time take $s_0=0$, $s_1=\frac{3}{8}$, $r_0=2$ and $r_1=4$. We can compute that $s_{\theta}-\frac{1}{r_{\theta}}=\frac{5\theta}{8}>0$, since $\theta = \frac{1+\delta - \varepsilon}{1+\delta + \varepsilon}$ is arbitrarily close to $1$.   
\end{proof}

\begin{proof}[Proof of Lemma~\ref{4lemmaTight}] We recall that $\|u\|_X= \inf \{\|u_1\|_{X_1}+\|u_2\|_{X_2}\}$ where the infimum runs over decompositions $u=u_1+u_2$ with $u_1 \in X_1$, and $u_2 \in X_2$. We write
\[u_{\nu} (t) = u_{\nu}(0)+ \int_0^t B(u_{\nu} (t'),u_{\nu}(t'))\,\mathrm{d}t' -\nu \int_0^t Lu_{\nu}(t') \,\mathrm{d}t' + \sqrt{\nu}\zeta(t)\,.\]
Then we set $u_{\nu}^{(1)}:=u_{\nu}(0)+ \displaystyle\int_0^t B(u_{\nu (t')},u_{\nu}(t'))\,\mathrm{d}t' -\nu \displaystyle\int_0^t Lu_{\nu}(t') \,\mathrm{d}t'$ and $u_{\nu}^{(2)}(t)=\sqrt{\nu}\zeta(t)$. 

The bounds in $L^2(\Omega,L^2([0,1],H^{2+\delta}))$ follow from Proposition~\ref{4propUnifBounds}, (\textit{ii}). For the other bounds we proceed as follows.

We observe that $\partial _t u_{\nu}^{(1)}(t)=B(u_{\nu} (t),u_{\nu}(t)) + Lu_{\nu}(t)$ thus $u_{\nu}$ is bounded in $L^2(\Omega, L^2(I,H^{-\delta}))$ thanks to Proposition~\ref{4propUnifBounds}, (\textit{ii}). 

It remains to prove that $u_{\nu}^{(2)}$ is bounded in $L^2(\Omega,W^{\frac{3}{8},4}([0,1],H^{1-2\varepsilon}))$, which in facts will result from proving that $\zeta \in L^2(\Omega, W^{\frac{3}{8},4}([0,1],H^{1-2\varepsilon})$. We compute:
\[
    \mathbb{E}[\|\zeta(t)\|^4_{H^{1-2\varepsilon}}] = \mathbb{E}\left[\left(\sum_{n \in \mathbb{Z}_0^2} \langle n\rangle ^{2(1-2\varepsilon)}|\phi_n|^2\beta_n(t)^2\right)^2\right] \leqslant \sum_{n \in \mathbb{Z}_0^2} \langle n \rangle ^4|\phi_n|^4\mathbb{E}[\beta_n(t)^4]\,.
\]
Now recall that by the properties of Gaussian variables,  $\mathbb{E}[\beta_n(t)^4] = 3t^2 \leqslant 3$ on the time-inverval~$I$ so that finally $\mathbb{E}[\|\zeta\|^4_{L^4(I,H^1)}] < \infty$,
since $\phi$ is a \textit{standard noise}. By the same techniques we obtain
\[\mathbb{E} \left[\int_{[0,1] \times [0,1]} \frac{\|\zeta (t)-\zeta (r)\|^4_{H^1}}{|t-r|^{1+4\alpha}}\,\mathrm{dr}\,\mathrm{d}t\right] < \infty\,,\]
and combined with the previous estimate we finally have $\mathbb{E}\left[\|u^{(2)}_{\nu}\|^2_{W^{\frac{3}{8},4}}\right] \lesssim \nu$.
\end{proof}

\subsection{Passing to the limit} 

Thanks to the previous subsection, we can assume that the sequence $\bar{\mu}_{\nu}$ converges as $\nu \to 0$. Indeed, by the Prokhorov theorem (see Theorem 11.5.4 in~\cite{dudley}), there exists a sequence $\nu _j \to 0$ and $\bar\mu \in \mathcal{P}\left(\mathcal{C}^0(\mathbb{R},H^{1-\varepsilon}) \cap L^{2+\delta-\varepsilon}\right)$ such that $\bar\mu_{\nu_j} \to \bar\mu$ in law. This means that $\mathbb{E}_{\bar\mu_{\nu_j}}[f] \to \mathbb{E}_{\bar \mu}[f]$ for any bounded Lipschitz function.  
    
Note that the measures $(\mu_{\nu _j})_{j \geqslant 0}$ also satisfy Prokhorov's theorem in $H^{2+\delta -\varepsilon}$, due to the uniform estimates in $H^{2+\delta}$, hence there is no loss of generality in assuming that $\mu_{\nu_j} \to \mu$ weakly in $\mathcal{P}(H^{2+\delta})$, where $\mu$ denotes the restriction at time $t=0$ of $\bar\mu$.  

By the Skorokhod theorem (see~\cite{dudley}, Theorem~11.7.2) there exists a sequence of $\bar \mu_{\nu}$-stationary processes $\tilde{u}_{\nu}$ and an $Y$-valued process $u$ defined on the same probability space such that
\[
\left\{
\begin{array}{cc}
    \mathcal{L}(\tilde{u}_{\nu_j})= \mu_{\nu_j} \text{ and } \mathcal{L}(u)=\mu \\
    \mathbb{P} \left(\tilde{u}_{\nu_j} \to u \text{ in } Y \text{ as } j\to \infty\right)=1\,,
\end{array}
\right.
\]
where $Y=\mathcal{C}^0(\mathbb{R},H^{1-\varepsilon}) \cap L^{2+\delta-\varepsilon}$.

We can further assume that the $\tilde{u}_{\nu _j}$ satisfy the same equation as $u_{\nu _j}$. More precisely, the $u_{\nu_j}$ satisfy an equation of the general form $F(u_{\nu _j})= \sqrt{\nu} \zeta$, where $F$ is a deterministic measurable functional and $\zeta$ is the random variable defined by~\eqref{4refWhite}. Since $\mathcal{L}(u_{\nu_j})=\mathcal{L}(\tilde{u}_{\nu_j})$, then $\mathcal{L}(F(\tilde{u}_{\nu _j})) = \mathcal{L}(F(u_{\nu _j})) = \mathcal{L}( \sqrt{\nu}\zeta)$ thus there exists a process $\tilde{\zeta}$ equal in law to $\tilde{\zeta}$ (and in particular share the same numbers $\phi_n$, that is there exists independent brownian motions $\tilde{\beta}_n$ such that $\tilde{\zeta} (t)= \sum_{n \in \mathbb{Z}^2_0} \phi_n \tilde{\beta}_n (t) e_n$) such that $F(\tilde{u}_{\nu_j})=\sqrt{\nu}\tilde{\zeta}$. In the rest of this section we will drop the $\tilde{}$ and write $u_{\nu_j}, \zeta$. 

We are ready to state the main results of this section. 

\begin{proposition}\label{4propStationaryMeasures} Let $\bar{\mu}, \mu$ and $u$ as defined above. Then we have,
\begin{enumerate}[label=(\roman*)]
    \item $\bar \mu(L^2_{\operatorname{loc}}(\mathbb{R}_+,H^{2+\delta}))=1$ and $\mu(H^{2+\delta})=1$.
    \item $\displaystyle \mathbb{E}_{\mu}\left[\|u\|_{\dot H^{2+\delta}}^2\right] \leqslant \frac{\mathcal{B}_1}{2}$ and $\displaystyle \mathbb{E}_{\mu}\left[e^{\gamma \|u\|_{\dot H^1}^2}\right] <\infty$.
    \item The process $u$ is stationary for the measure $\mu$ and satisfies the Euler equation~\eqref{4Euler2}.
    \item The constant $C$ defined by $\mathbb{E}_{\mu}\left[\|u\|^2_{\dot H^1}\right] = C$ satisfies: 
    \begin{equation}
        \label{4eq:encadrementC}
        C_{\delta} \coloneqq \frac{1}{2}\left(\frac{\mathcal{B}_0}{\mathcal{B}_1^{\frac{\delta}{1+\delta}}}\right)^{1+\delta}\leqslant C \leqslant \frac{\mathcal{B}_0}{2}\,. 
    \end{equation}
    As a consequence $\mathbb{E}_{\mu}\left[\|u\|_{\dot H^{2+\delta}}^2\right] \in [C_{\delta},\frac{\mathcal{B}_1}{2}]$.
\end{enumerate}
\end{proposition}*
\begin{remark} The measure $\mu$ also depend on $\delta$, as it is obtained as a limit measures from $\mu_{\nu}=\mu_{\nu}^{(\delta)}$.  
\end{remark}

\begin{proof} 
(\textit{i}) Let $T>0$. Thanks to Proposition~\ref{4propUnifBounds} we have $\mathbb{E}_{\mu_{\nu}}[\|u\|^2_{L^2([0,T],H^{2+\delta})}] \leqslant C$ independently of $\nu$. Note that we also have, uniformly in $\nu$, $M$ and $K$:
\[\mathbb{E}_{\mu_{\nu_j}}\left[\min\left\{\|\mathbf{P}_{\leqslant M}u\|^2_{L^2([0,T],H^{2+\delta})},K\right\}\right] \leqslant C\,.\] 
Since the map map $u \mapsto \min \left\{\|\mathbf{P}_{\leqslant M}u\|^2_{H^{2+\delta}},K\right\}$ is continuous and bounded as a map $H^{2+\delta - \varepsilon} \to \mathbb{R}$ we can pass to the limit $j \to \infty$ by weak convergence of the measures and get 
\[
    \mathbb{E}_{\mu}\left[\min \left\{\| \mathbf{P}_{\leqslant M}u\|^2_{L^2([0,T],H^{2+\delta})},K\right\}\right] \leqslant C\,,
\]
and conclude by monotone convergence as $K \to \infty$ and then $M \to \infty$.  

(\textit{ii}) follows by a similar argument. 

(\textit{iv}) First, we claim that for any $j$ the following uniform inequalities hold:
\begin{equation}
    \label{4uniform}
    C_{\delta} \leqslant \mathbb{E}_{\mu_{\nu_j}}\left[\|u\|_{\dot H^1}^2\right]\leqslant \frac{\mathcal{B}_0}{2}
\end{equation}
In order to prove~\eqref{4uniform}, we observe that the right-hand side inequality is contained in Proposition~\ref{4propUnifBounds}~(\textit{i}). For the lower bound, we start by using the interpolation inequality
\[
    \|u\|_{\dot H^{1+\delta}}^2 \leqslant \|u\|_{\dot{H}^1}^{\frac{2}{1+\delta}}\|u\|_{\dot H^{2+\delta}}^{\frac{2\delta}{1+\delta}}\,,
\]
and by the Hölder inequality we infer 
\[
    \mathbb{E}_{\mu_{\nu_j}}[\|u\|_{\dot H^{1+\delta}}^2] \leqslant \mathbb{E}_{\mu_{\nu_j}}\left[\|u\|_{\dot{H}^1}^2\right]^{\frac{1}{1+\delta}}\mathbb{E}_{\mu_{\nu_j}}\left[\|u\|_{\dot H^{2+\delta}}^{2}\right]^{\frac{\delta}{1+\delta}}\,.
\]
The proof of the claim now follows from Proposition~\ref{4propUnifBounds} (\textit{i}) and (\textit{ii}). 

We claim that there exists a constant $C>0$ such that
\begin{equation}
    \label{4eq.ClaimLimit}
    \limsup_{j \geqslant 0} \int_{\|u\|_{\dot{H}^1}>R} \|u\|_{\dot{H}^1}^2\,\mu_{\nu_j}(\mathrm{d}u) \leqslant \frac{C}{R}\,.
\end{equation}
Assuming this claim we can finish the proof of (\textit{iv}). To this end, let $\chi$ be a smooth bump function, such that $\chi =1$ on $[-1,1]$, vanishing outside $[-2,2]$ and let $\chi_R=\chi(\cdot /R)$. Let $\varepsilon>0$ and $R>0$ be fixed. Then by \eqref{4eq.ClaimLimit} there exists $J=J(R)>0$ large enough such that for any $j \geqslant J$ there holds
\[ \int_{\|u\|_{\dot{H}^1} \leqslant R} \|u\|_{\dot H^1}^2\,\mu_{\nu_j}(\mathrm{d}u) \geqslant  C_{\delta}- \varepsilon\,.\]
Now observe that $\chi_{\frac{R}{2}}(\|u\|_{\dot H^1}) \leqslant \mathbf{1}_{\|u\|_{\dot H^1}>R}$ hence
\[ \int \chi_{\frac{R}{2}}(\|u\|_{\dot H ^1})\|u\|_{\dot H^1}^2\,\mu_{\nu_j}(\mathrm{d}u) \geqslant  C_{\delta}- \varepsilon\,.\]
Then passing to the limit $j \to \infty$, and because $u \mapsto \chi_{\frac{R}{2}}(\|u\|_{\dot H ^1})\|u\|_{\dot H^1}^2$ is a continuous bounded map $\dot H^1 \to \mathbb{R}$, we obtain
\[\int_{\|u\|_{\dot{H}^1} \leqslant R} \|u\|_{\dot H^1}^2\,\mu (\mathrm{d}u) \geqslant  C_{\delta}- \varepsilon \,,\]
and finally the monotone convergence theorem in the limit $R \to \infty$ proves~\eqref{4eq:encadrementC}. 

It remains to establish~\eqref{4eq.ClaimLimit}. In order to do so, we remark that the Cauchy-Schwarz inequality followed by the Markov inequality yields:
\begin{align*}
    \int_{\|u\|_{\dot{H}^1}>R} \|u\|_{\dot{H}^1}^2\,\mu_{\nu_j}(\mathrm{d}u) &\leqslant  \mu_{\nu_j} (\|u\|_{\dot{H}^1}>R)^{\frac{1}{2}} \left(\int_{\|u\|_{\dot{H}^1}>R} \|u\|_{\dot{H}^1}^4\,\mu_{\nu_j}(\mathrm{d}u)\right)^{\frac{1}{2}}\\
    & \leqslant \frac{1}{R} \mathbb{E}_{\mu_{\nu _j}}\left[\|u\|^2_{\dot{H}^1}\right]^{\frac{1}{2}} \mathbb{E}_{\mu_{\nu _j}}\left[\|u\|^4_{\dot{H}^1}\right]^{\frac{1}{2}}\,.
\end{align*}
Now remark that there is a constant $C=C(\gamma)>0$ such that for any real numbers $x$ there holds $x^4 \leqslant Ce^{\gamma x^2}$ so that we obtain
\begin{align*}
    \int_{\|u\|_{\dot{H}^1}>R} \|u\|_{\dot{H}^1}^2\,\mu_{\nu_j}(\mathrm{d}u) &\leqslant \frac{C}{R} \mathbb{E}_{\mu_{\nu _j}}\left[\|u\|^2_{\dot{H}^1}\right]^{\frac{1}{2}} \mathbb{E}_{\mu_{\nu _j}}\left[e^{\gamma\|u\|^2_{\dot{H}^1}}\right]^{\frac{1}{2}} \\
    &\leqslant \frac{C}{R}\,,
\end{align*}
where in the last line we used the uniform estimates of Proposition~\ref{4propUnifBounds}~(\textit{iii}). This establishes~\eqref{4eq.ClaimLimit}.  

(\textit{iii}) Invariance follows from the weak convergence of the measures. It remains to prove that~$u$ satisfies the Euler equation on any time-interval $[0,T]$. We start by writing that  
\begin{equation}
    \label{4eqNSnulimit}
    u_{\nu}(t)=e^{-\nu t L}u_0 + \int_0^t e^{-\nu (t-t')L}\left(B(u_{\nu} (t'),u_{\nu} (t'))+\sqrt{\nu}\zeta (t')\right)\,\mathrm{d}t'\,.
\end{equation}
As we want to pass to the limit in $L^2$, we use the previous uniform bounds. First, we have $e^{-t\nu L}u_0 \to u_0$ in $L^{\infty}((0,T),L^2)$ and also $u_{\nu} \to u$ in $L^{\infty}((0,T),L^2)$. Next, we use the triangle inequality to bound: 
\[\sqrt{\nu}\left\| \int_0^t e^{-\nu (t-t')L}\zeta (t')\,\mathrm{d}t'\right\|_{L^2} \leqslant \sqrt{\nu}\int_0^t \|\zeta(t')\|_{L^{2}}\,\mathrm{d}t'\,,\]
We have seen in the previous subsection that $\mathbb{E}[\|\zeta\|_{L^4((0,T),H^{1-2\varepsilon}}] < \infty$, which implies in particular that $\sqrt{\nu} \zeta \to 0$ in probability, in the space $L^4((0,T),L^2)$, thus up to passing to a sub-sequence this yields
\[\mathbb{P}\left(\lim_{j \to \infty} \sqrt{\nu_j}\zeta =0 \text{ in } L^4((0,T),L^2)\right)=1\,,\]
and then for $t \in [0,T]$:
\[\sqrt{\nu_j}\left\| \int_0^t e^{-\nu_j (t-t')L}\zeta (t')\,\mathrm{d}t'\right\|_{L^2}\leqslant T^{3/4}\|\sqrt{\nu_j}\zeta\|_{L^4((0,T),L^2)} \underset{j \to \infty}{\longrightarrow} 0\,.\]
To bound the remaining term, we use Lemma~\ref{4lemmaHeat} and write:
\begin{align*}
 \Big\lVert\int_0^t e^{-\nu(t-t')L}\left(B(u_{\nu}(t'),u_{\nu}(t'))- \right.& \left. B(u(t'),u(t'))\right)\,\mathrm{d}t'\Big\rVert_{L^2}\\ 
&\lesssim \int_0^T \frac{1}{\sqrt{t'}} \|B(u_{\nu}(t'),u_{\nu}(t'))-B(u(t'),u(t'))\|_{H^{-1}}\,\mathrm{d}t'\\
&\lesssim \|B(u_{\nu},u_{\nu})-B(u,u)\|_{L^{\infty}((0,T),H^{-1})}\\
& \lesssim \|u_{\nu}-u\|_{L^{\infty}((0,T),H^{\frac{1}{2}})}\|u_{\nu}\|_{L^{\infty}((0,T),H^{\frac{1}{2}})}\,.
\end{align*}
Since $(u_{\nu_j})_{j\geqslant 0}$ is convergent to $u$ (and therefore bounded) in $L^{\infty}((0,T),H^{1-\varepsilon})$ we obtain that for any $t \in [0,T]$:
\begin{align*}
 \left\|\int_0^t e^{-\nu_j(t-t')L}\left(B(u_{\nu_j}(t'),u_{\nu_j}(t'))- B(u(t'),u(t'))\right)\,\mathrm{d}t'\right\|_{L^2}\underset{j \to \infty}{\longrightarrow}0\,.
\end{align*}
Finally, the process $u$ almost surely satisfies $u \in \mathcal{C}^0(\mathbb{R}_+,H^1_{\operatorname{div}}) \cap  L^2_{\operatorname{loc}}(\mathbb{R}_+,H^{2+\delta})$ and obeys:
\[u(t)=u_0+\int_0^t \mathbf{P}\left(u(t')\cdot \nabla u(t')\right)\,\mathrm{d}t'\,,\]
where the equality holds in $\mathcal{C}^0(\mathbb{R}_+,L^2)$,
hence $u$ is a solution to the Euler equation~\eqref{4Euler2}.
\end{proof}

\section{Proof of the main results}\label{4sec5}

\subsection{About the classical local well-posedness theory for the Euler equations}

We will use that the local well-posedness time only depends on the size of the initial data. The following result is well-known:  

\begin{lemma}\label{4lemmaLocalEuler} Let $s>2$. Then the Euler equation~\eqref{4Euler2} is globally well-posed in $\mathcal{C}^0(\mathbb{R}_+,H^{s}(\mathbb{T}^2))$. Moreover, there exists $c>0$ such that if $u_0 \in H^s(\mathbb{T}^2)$ and $\tau \coloneqq \frac{c}{\|u_0\|_{H^s}}$, then there holds 
\[\|u(t)\|_{H^s} \leqslant 2 \|u_0\|_{H^s}\,,\]
for any $t \in [0,\tau]$. 
\end{lemma}

\begin{proof} The local well-posedness statement follows from standard approximation arguments based on the following \textit{a priori} estimate:
\[
    \frac{\mathrm{d}}{\mathrm{d}t}\|u(t)\|^2_{H^s} \lesssim \|\nabla u(t)\|_{L^{\infty}}\|u(t)\|_{H^s}^2 \lesssim \|u(t)\|_{H^s}^3\,,
\]
where we used the Sobolev embedding $H^{s-1} \hookrightarrow L^{\infty}$. Integrating this inequality, there exists $C>0$ such that
\[\|u(t)\|_{H^s} \leqslant  \frac{\|u_0\|_{H^s}}{1-Ct\|u_0\|_{H^s}}\,,\]
as long as $1-Ct\|u_0\|_{H^s}>0$. Setting $\tau \coloneqq \frac{1}{2C\|u_0\|_{H^s}}$ yields the result. 

Finally the global well-posedness can be obtained using the blow-up criterion of Beale-Kato-Majda, which states that if we denote $T^*$ the maximal existence time, then: 
\begin{equation}
\label{4bkm}
    T^*<\infty \Longrightarrow \int_0^{T^*}\|\xi(t)\|_{L^{\infty}}\,\mathrm{d}t = \infty\,.
\end{equation}
For a proof, see~\cite{bkm}. Since the vorticity satisfies~\eqref{4EulerVorticity}, we see that $\xi$ is transported by the flow, thus its $L^{\infty}$ norm is conserved, so that if we assume that $T^*<\infty$, then 
\[\displaystyle \int_0^{T^*}\|\xi(t)\|_{L^{\infty}}\,\mathrm{d}t=T^*\|\xi_0\|_{L^{\infty}}< \infty\] which is a contradiction, thus $T^*=\infty$. 
\end{proof}

First we remark that thanks to the global well-posedness result for two-dimensional Euler equations we have the following result. 

\begin{proposition}\label{4propInvariant2D} The measure $\mu$ is invariant under the well-defined flow $\Phi_t : H^{2+\delta} \to H^{2+\delta}$ of~\eqref{4Euler2}. More precisely, for any $A \in \mathcal{B}(H^{2+\delta})$ and every $t>0$ there holds: \[\mu(\Phi_{-t}(A))=\mu(A)\,.\]
\end{proposition}

\begin{proof}
This result will follow from the stationarity of the process $u$ with respect to $\mu$ once we prove that $\mu$ almost-surely, $u(t)=\Phi_t(u_0)$. It is a weak-strong uniqueness statement. Let $u_0 \in H^1$. As we have already seen it, $\mu$-almost surely, $u_0\in H^{2+\delta}$, so that there is no loss of generality assuming that $u_0 \in H^{2+\delta}$. Then, let $v=\Phi_t(u_0)$ the global solution to~\eqref{4Euler2} associated to $u_0$ which is constructed by Lemma~\ref{4lemmaLocalEuler}. Let us denote by $u$ the solution constructed by Proposition~\ref{4propUnifBounds}, associated to the initial data $u_0$ and solution to~\eqref{4Euler2} in the class $\mathcal{C}^0H^{1-\varepsilon} \cap L^2_{\operatorname{loc}}H^{2+\delta - \varepsilon}$. Let $T>0$. We want to prove that $u=v$ on $[0,T]$. We introduce $w:=u-v$, and then we compute 
\[
    \frac{\mathrm{d}}{\mathrm{d}t} \|w(t)\|_{L^2}^2 =-2 (w\cdot \nabla u,w)_{L^2} \leqslant 2\|w\|_{L^2}^2\|u\|_{H^{2+\delta-\varepsilon}}\,,
\] 
where we used the properties of the bilinear form and the Sobolev embedding. Since we have $u \in L^2([0,T],H^{2+\delta-\varepsilon}) \subset L^1([0,T],H^{2+\delta-\varepsilon})$, and $w(0)=0$, the Grönwall lemma implies that~$w=0$ on $[0,T]$, hence the conclusion. 
\end{proof}

Before the proof of Theorem~\ref{4sideTheo2d}, let us remark that from the bounds on $\|u\|_{\dot H^1}$ and $\|u\|_{\dot H^{2+\delta}}$ we can infer the following interpolation inequality. 

\begin{lemma}\label{4interpolation} Let $\sigma \in (1,2+\delta]$, then: 
\[
    \mathbb{P}\left( \|u\|_{H^\sigma}>\lambda \right) \lesssim  \lambda^{-\frac{2(1+\delta)}{\sigma-1}}\log^{\beta (\sigma)} (\lambda)\,,
\]
with $\beta(\sigma)=\frac{2+\delta -\sigma}{\sigma -1}$ and where the implicit constant depends only on the constants $C_1, C_2$ and $\sigma$. 
\end{lemma}

We postpone the proof of this technical estimate to the end of the section and proceed to the proof of the main theorems. First, note that Theorem~\ref{4mainTheo2d3d} (\textit{i}), (\textit{ii}) in the two-dimensional case is already proven since we have constructed an invariant measure $\mu$ in Section~\ref{4sec4} and its properties will be studied in Section~\ref{4sec6}. It remains to prove Theorem~\ref{4sideTheo2d} and Theorem~\ref{4mainTheo2d3d}~(\textit{iii}).    

\begin{proof}[Proof of Theorem~\ref{4sideTheo2d} (\textit{i})] Given that $\mu$ is an invariant measure for~\eqref{4Euler2}, we are going to prove the theorem using the same argument as in~\cite{bourgain}, explained in the introduction. Let $\sigma \in (1,s]$ where $s=2+\delta$. Let $T>0$ and $\varepsilon >0$. Let us define
\begin{equation}
    \label{eq.lambda}
    \lambda\coloneqq \lambda(\varepsilon, T)= (T\varepsilon^{-1})^{\frac{\sigma - 1}{3-\sigma + 2\delta}}\,.
\end{equation}

Lemma~\ref{4lemmaLocalEuler} provides us with $\tau=\frac{c_1}{\lambda}$ such that if $\|u_0\|_{H^{\sigma}} \leqslant \lambda$, then there is a unique local solution $u(t)$ to~\eqref{4Euler2} in $\mathcal{C}^0([0,\tau],H^{\sigma})$ obeying $\|u(t)\|_{H^{\sigma}} \leqslant 2\lambda$ for any $t \in [0,\tau]$. 

Let us consider
\[
    G_{\varepsilon, T}\coloneqq \bigcap _{n=0}^{\left\lfloor \frac{T}{\tau}\right\rfloor} \{u_0 \in H^{\sigma} \,, \Phi_{n \tau}(u_0) \in B_{\lambda}\}\,,
\]
where $B_{\lambda}=B_{H^{\sigma}}(0,\lambda)$. 

Thanks to the definition of $\tau$ we infer that for all $u_0\in G_{\varepsilon, T}$ there holds $\|u(t)\|_{H^{\sigma}}\leqslant 2\lambda$ for all $t \in [0,T]$. Then we compute, using the invariance of the measure on $H^s$:
\begin{align*}
    \mu\left(H^{\sigma}\setminus G_{\varepsilon, T}\right) &\leqslant \sum_{n=0}^{\left\lfloor \frac{T}{\tau}\right\rfloor } \mu \left(\Phi_{n \tau} \left(H^{\sigma} \setminus B_{\lambda}\right)\right) \\
    & \leqslant \sum_{n=0}^{\left\lfloor \frac{T}{\tau}\right\rfloor } \mu \left(H^{\sigma} \setminus B_{\lambda}\right) \\
    & \leqslant T \tau^{-1}\mu\left(H^{\sigma} \setminus B_{\lambda}\right)\,.
\end{align*}
We now use that $\tau^{-1}=\frac{1}{c_1}\lambda$, Lemma~\ref{4interpolation} to obtain and~\eqref{eq.lambda} to obtain 
\[
    \mu(H^{\sigma}\setminus G_{\varepsilon, T}) \leqslant  T\lambda^{\frac{{\sigma}-3-2\delta}{{\sigma}-1}} \log^{\beta ({\sigma})}\lambda \leqslant  \varepsilon \log^{\beta(\sigma)}(T\varepsilon^{-1})\,. 
\]

The remainder of the proof is now a Borel-Cantelli argument: set $T_n\coloneqq 2^n$ and $\varepsilon_n \coloneqq n^{-(2+\beta(\sigma))}$. Then consider 
\[
    G\coloneqq \bigcup_{n \geqslant 1} \bigcap_{k \geqslant n} G_{\varepsilon_k,T_k}\,,
\] 
which satisfies $\mu(G)=1$ because $\mu(H^{\sigma}\setminus G_{\varepsilon_n, T_n}) \leqslant C n^{-2}$, which is summable. Moreover, for any $u_0 \in G$ there is a $n_0=n_0(u_0, \sigma, s)$ such that we have $u_0 \in G_{n_0}$, which gives for all $k\geqslant n_0$ and all $t \in [T_{k-1}, T_k]$, 
\[
    \|u(t)\|_{H^{\sigma}} \leqslant 2 (T_k\varepsilon^{-1}_k)^{\frac{\sigma - 1}{3-\sigma +2\delta}} \leqslant C t ^{\frac{\sigma -1}{2s-\sigma -1}} (\log t)^{\frac{\beta(\sigma)(\sigma -1)}{2s-\sigma -1}} \leqslant C t ^{\frac{s- \sigma }{2s-\sigma -1}} (\log t) \leqslant C t^{\alpha} \,,
\]
for any $\alpha > \frac{\sigma - 1}{2s-1-\sigma}$. This concludes the proof. 
\end{proof}

\begin{proof}[Proof of Theorem~\ref{4sideTheo2d} (ii)] 
It is a direct consequence of (\textit{i}) and the use of the Sobolev embedding $H^{\sigma} \mapsto W^{2,\infty}$ as soon as $\sigma>3$: 
\[
    \|\nabla \xi (t)\|_{L^{\infty}} \lesssim \|\xi(t)\|_{H^{\sigma -1}} \lesssim \|u(t)\|_{H^{\sigma}}\,.\qedhere
\]
\end{proof}

\begin{proof}[Proof of Theorem~\ref{4mainTheo2d3d}] Part (\textit{i}) and (\textit{ii}) have already been proven. 

(\textit{iii}) Theorem~\ref{4mainCoro2d} implies in particular that such measures cannot have any atom: indeed, if $u^0$ is an atom, and $c\coloneqq \|u^0\|_{L^2}$ then $\mu\left(u \in H^{2+\delta}\text{ such that } \|u\|_{L^2} =c\right)>0$, but also 
\[\mu \left(u \in H^{2+\delta}\text{ such that } \|u\|_{L^2} \in (c-\varepsilon, c+\varepsilon)\right) \leqslant p(2\varepsilon) \underset{\varepsilon \to 0}{\longrightarrow}0\,,\]
which is a contradiction.  

(\textit{iv}) Let $k>0$. We claim that there exists a choice of the noise parameters $(\phi_n)_{n\in\mathbb{Z}^2_0}$ such that the associated invariant measure $\mu^{(k)}$ satisfies
\begin{equation}
    \label{4eq:noiseparamChoiceK}
    \mathbb{E}_{\mu^{(k)}}[\|u\|_{\dot H^1}^2] \sim k^{\alpha (\delta)}\,,
\end{equation}
for some $\alpha (\delta)>0$. Assuming~\eqref{4eq:noiseparamChoiceK}, one can finish the proof by letting 
$\mu \coloneqq \sum_{k \geqslant 1} \frac{\mu ^{(k)}}{2^k}$
which is an invariant probability measure by construction. Moreover, observe that for any $R>0$ we have $\mathbb{P}(\|u\|_{\dot{H}^{2+\delta}}^2 >R)>0$, completing the proof of (\textit{iv}). Indeed, if there exists $R_0>0$ such that $\mathbb{P}(\|u\|_{\dot{H}^{2+\delta}}^2 \leqslant R_0)=1$ then for $k>R_0$ we have $R_0^2 \geqslant \mathbb{E}_{\mu^k}[\|u\|_{\dot H^{2+\delta}}^2] \sim k^{\alpha (\delta)} \to \infty$
as $k \to \infty$, which is a contradiction. 

It remains to prove~\eqref{4eq:noiseparamChoiceK}. In view of~\eqref{4eq:encadrementC} the matter reduces to choosing the $\phi_n$ such that 
\[
    \frac{\mathcal{B}_0}{\mathcal{B}_1^{\frac{\delta}{1+\delta}}} \gtrsim_{\delta} K^{\varepsilon}\,,
\]
for some $\varepsilon>0$. Take $\alpha \geqslant \frac{\varepsilon(1-\delta)}{2} + \delta$ and define
\[
\phi_n= \left\{
\begin{array}{cc}
    e^{-n} & \text{ if } |n|<k \\
    \frac{k^{\alpha}}{n} & \text{ if } k\leqslant |n| \leqslant 2k\\
    e^{-(n-2k)} & \text{ if } |n|>2k\,.
\end{array}
\right.
\]
Observe that 
\begin{align*}
  \mathcal{B}_0 &= \sum_{|n|<k}|\phi_n|^2 + \sum_{k\leqslant |n| \leqslant 2k} |\phi_n|^2 + \sum_{|n|>2k} |\phi_n|^2 \\
  & = \mathcal{O}(1) + k^{2\alpha}\mathcal{O}(1) + \mathcal{O}(1) \\
  & \sim k^{2\alpha}\,,  
\end{align*}
and similarly $\mathcal{B}_1 \sim k^{2\alpha+2}$. Therefore, thanks to the assumption on $\alpha$ we deduce~\eqref{4eq:noiseparamChoiceK}. 
\end{proof}

To complete this paragraph we give a proof of Lemma~\ref{4interpolation}. 

\begin{proof}[Proof of Lemma~\ref{4interpolation}] 
We use an elementary argument : let $A>0$ be a parameter to be chosen later. Since for $\sigma\in (1,2+\delta]$ we have the interpolation inequality 
\[
    \|u\|_{H^\sigma} \leqslant \|u\|_{H^1}^{\theta}\|u\|_{H^{2+\delta}}^{1-\theta}
\] 
where $\theta\coloneqq \frac{2+\delta -\sigma}{1+\delta}$, we write that for any $\lambda >0$ and $A>0$: 
\[
    \{\|u\|_{H^\sigma} > \lambda \} \subset \{\|u\|_{H^1}>A^{-\frac{1}{\theta}}\lambda\} \cup \{\|u\|_{H^2}>A^{\frac{1}{1-\theta}}\lambda\}\,,
\]
so that the Markov inequality implies the bound
\begin{align*}
    \mathbb{P}(\|u\|_{H^\sigma}>\lambda)&\leqslant \mathbb{P}\left( e^{\gamma\|u\|_{H^1}^2}>\exp\left(\gamma A^{-\frac{2}{\theta}}\lambda^2\right)\right) + \mathbb{P}\left(\|u\|_{H^{2+\delta}}^2 > A^{\frac{2}{1-\theta}}\lambda^2\right) \\
    & \lesssim \exp\left(- \gamma A^{-\frac{2}{\theta}}\lambda^2\right) + A^{-\frac{2}{1-\theta}}\lambda^{-2}\,.
\end{align*}
Next, we want to optimise in $A$ in the right-hand side so that both terms have the same size, \textit{i.e.}, $\exp\left(-\gamma A^{-\frac{2}{\theta}}\lambda^2\right) \simeq A^{-\frac{2}{1-\theta}}\lambda^{-2}$, or taking the logarithm twice in this relation gives $\frac{2}{\theta} \log A \sim 2 \log \lambda$ and re-plugging in the previous relation suggests the choice $A\coloneqq \left(\frac{2}{\gamma(1-\theta)}\right)^{-\frac{\theta}{2}} \frac{}{}\lambda^{\theta}\log^{-\frac{\theta}{2}}(\lambda)$, leading to 
\[
    \mathbb{P}\left(\|u\|_{H^s}>\lambda \right) \lesssim \lambda^{-\frac{2}{1-\theta}}\log^{\frac{\theta}{1-\theta}}(\lambda)\,,
\]
which concludes the proof recalling the definition of $\theta$. 
\end{proof}

\section{Properties of the measures}\label{4sec6}

The goal of this section is to prove Theorem~\ref{4mainCoro2d}. Note that thanks to the Portmanteau theorem and the interior regularity of the Lebesgue measure, it is sufficient to prove the following proposition for the measures $\mu_{\nu}$. 

\begin{proposition}\label{4propGoodMeasures} The measures $(\mu_{\nu})_{\nu >0}$ satisfy the following properties. 
\begin{enumerate}[label=(\textit{\roman*})]
    \item There exists a constant $C>0$ independent of $\nu$ such that for all $\varepsilon>0$ there hods:
    \[\mu_{\nu} \left(\|u\|_{L^2}<\varepsilon\right) \leqslant C \varepsilon^{\frac{2(1+\delta)}{2+\delta}}\,.\] 
    \item There exists a continuous increasing function $p : \mathbb{R}_+ \to \mathbb{R}_+$ and a constant $C>0$ which does not depend on $\nu$ such that for every Borel set $\Gamma \subset \mathbb{R}_+$ there holds
    \[\mu_{\nu} \left(\|u\|_{L^2} \in \Gamma\right) \leqslant p(|\Gamma|)\,.\]
\end{enumerate}
\end{proposition}

The proof will heavily rely on Proposition~\ref{4propFormula} whose proof is given in Appendix~\ref{4appA}. The proof is taken from~\cite{kuksin} with some minor modifications designed to adapt it to the case of the article.   

\begin{proof}[Proof of Proposition~\ref{4propGoodMeasures}] We start with (\textit{i}). First, we prove that
\begin{equation}
    \label{4eqLebesgue1}
    \mu \left(\{u \in L^2, \; 0<\|u\|_{L^2} \leqslant \varepsilon\}\right) \lesssim \varepsilon^{\frac{2(1+\delta)}{2+\delta}}\,.
\end{equation}

We let $u(t)=u_{\nu}(t)$ be a stationary process for $\mu_{\nu}$ satisfying~\eqref{4eqNSv2}. We apply Proposition~\ref{4propFormula} to $\Gamma = [\alpha, \beta]$ where $\alpha >0$ and $g \in \mathcal{C}^2(\mathbb{R})$ is a function such that $g(x)=x^{\frac{1+\delta}{2+\delta}}$ for $x \geqslant \alpha$ and vanishes for $x \leqslant 0$. This results in 
\begin{align*}
    \mathbb{E}&\int_{\alpha}^{\beta} \mathbf{1}_{(a,\infty)}\left(\|u\|_{L^2}^{\frac{2(1+\delta)}{2+\delta}}\right) \left(\frac{1+\delta}{2+\delta}\frac{\mathcal{B}_0/2-\|\nabla^{1+\delta} u\|_{L^2}^2}{\|u\|_{L^2}^{\frac{2}{2+\delta}}} - \frac{(1+\delta)\displaystyle\sum_{n\in \mathbb{Z}_0^2} |\phi_n|^2|u_n|^2}{(2+\delta)^2\|u\|_{L^2}^{\frac{2(3+\delta)}{2+\delta}}}\right) \,\mathrm{d}a \\
    & + C(\delta) \sum_{n \in \mathbb{Z}_0^2} |\phi_n|^2 \mathbb{E}\left[\mathbf{1}_{[\alpha,\beta]}\left(\|u\|_{L^2}^{\frac{2(1+\delta)}{2+\delta}}\right)\|u\|_{L^2}^{-\frac{4}{2+\delta}}u_n^2\right]=0\,.
\end{align*}
In particular this gives
\[\mathbb{E}\int_{\alpha}^{\beta} \mathbf{1}_{(a,\infty)}\left(\|u\|_{L^2}^{\frac{2(1+\delta)}{2+\delta}}\right) \left(\frac{1+\delta}{2+\delta}\frac{\mathcal{B}_0/2-\|\nabla^{1+\delta} u\|_{L^2}^2}{\|u\|_{L^2}^{\frac{2}{2+\delta}}} - \frac{(1+\delta)\displaystyle\sum_{n \in \mathbb{Z}_0^2}|\phi_n^2||u_n^2|}{(2+\delta)^2\|u\|_{L^2}^{\frac{2(3+\delta)}{2+\delta}}}\right) \,\mathrm{d}a\leqslant 0\,,\]
which gives
\begin{equation}
    \label{4eqMajo}
    \mathbb{E}\int_{\alpha}^{\beta} \frac{\mathbf{1}_{(a,\infty)}\left(\|u\|_{L^2}^{\frac{2(1+\delta)}{2+\delta}}\right)}{\|u\|_{L^2}^{\frac{2(3+\delta)}{2+\delta}}} \left(\mathcal{B}_0(1+\frac{\delta}{2})\|u\|_{L^2}^2-\sum_{n\in\mathbb{Z}_0^2} |\phi_n|^2u_n^2\right)\,\mathrm{d}a\lesssim (\beta - \alpha) \mathbb{E}\left[\frac{\|u\|_{\dot H^{1+\delta}}^2}{\|u\|_{L^2}^{\frac{2}{2+\delta}}}\right]\,.
\end{equation}
Observe that by interpolation 
\[\mathbb{E}\left[\frac{\|u\|_{\dot H^{1+\delta}}^2}{\|u\|_{L^2}^{\frac{2}{2+\delta}}}\right] \lesssim \mathbb{E}\left[\|u\|_{\dot H^{2+\delta}}^{\frac{2(1+\delta)}{2+\delta}}\right]\]
and 
\[
    \mathcal{B}_0(1+\frac{\delta}{2})\|u\|_{L^2}^2-\sum_{n\in\mathbb{Z}_0^2} |\phi_n|^2u_n^2 \geqslant \frac{\delta \mathcal{B}_0}{2}\|u\|_{L^2}\,.
\]
Plugging it into~\eqref{4eqMajo} provides:
\[\mathbb{E}\left[\int_{\alpha}^{\beta}\mathbf{1}_{(a,\infty)} \left(\|u\|_{L^2}^{\frac{2(1+\delta)}{2+\delta}}\right)\|u\|_{L^2}^{-\frac{2}{2+\delta}} \,\mathrm{d}a\right] \lesssim (\beta-\alpha) \lesssim \beta\,,\]
with constants independent of $\nu$. Then for any $\varepsilon >\beta$ one has 
\begin{align*}
    \mathbb{E}\left[\int_{\alpha}^{\beta} \mathbf{1}_{(a,\infty)}\left(\|u\|_{L^2}^{\frac{2(1+\delta)}{2+\delta}}\right)\|u\|_{L^2}^{-\frac{2}{2+\delta}} \, \mathrm{d}a\right] &\geqslant \varepsilon ^{-\frac{2}{2+\delta}} \mathbb{E}\left[\int_0^{\beta} \mathbf{1}_{(a,\varepsilon)}\left(\|u\|_{L^2}^{\frac{2(1+\delta)}{2+\delta}}\right)\,\mathrm{d}a\right] \\
    &= \varepsilon ^{-\frac{2}{2+\delta}} \int_0^{\beta} \mathbb{P}\left(a^{\frac{2+\delta}{2(1+\delta)}}< \|u\|_{L^2} < \varepsilon ^{\frac{2+\delta}{2(1+\delta)}}\right)\,\mathrm{d}a\,.
\end{align*}
so that finally,
\[\frac{1}{\beta}\int_0^{\beta} \mathbb{P}\left(a^{\frac{2+\delta}{2+2\delta}}< \|u\|_{L^2} < \varepsilon ^{\frac{2+\delta}{2+2\delta}}\right) \,\mathrm{d}a \lesssim \varepsilon^{\frac{2}{2+\delta}}\,,\]
and then passing to the limit $\beta \to 0$, one gets \eqref{4eqLebesgue1}.  

In order to complete the proof of (\textit{i}) we need to prove that $\mu_{\nu}$ does not have any atom at $0$. We refer to~\cite{shirikyan11} for details of this proof and also~\cite{kuksin} for an alternative proof. Here the dependence on $\nu$ is harmless, hence we drop the subscripts. Let $\mu_j$ be the law of the random variable $u_j(t)\coloneqq (u(t), e_j)_{L^2}$. We will prove that $\mu_j$ has no atom at zero, for any $j$, thus proving the desired fact. The Itô formula reads:
\[
    u_j(t)-u_j(0)=\int_0^th_j(s)\,\mathrm{d}s+\sqrt{\nu} \sum_{n \in \mathbb{Z}_0^2}\phi_n \mathrm{d}\beta _n\,,
\]
where $h_j(s)=(e_j,-B(u(s),u(s))+\nu Lu(s))_{L^2}$. Now observe that an application of Proposition~\ref{4propFormula} gives the existence of a random time $\Lambda_t(a)$ satisfying 
\[
    \Lambda_t(a)=|u_j(t)-a|-|u_j(0)-a|-\int_0^t \mathbf{1}_{(a,\infty)}(u_j(s))h_j(s)\,\mathrm{d}s - \sqrt{\nu} \sum_{n \in \mathbb{Z}_0^2} \phi_n \int_0^t \mathbf{1}_{(a,\infty)}(u_j(s))\,\mathrm{d}\beta_n(s)\,.
\]
Taking the expectation and using invariance yields $\mathbb{E}[\Lambda_t(a)]=-t\mathbb{E}[\mathbf{1}{(a,\infty)}(u_j(0))h_j(0)]$ and Theorem~\ref{4theoLocalTime} implies 
\[2\int_{\Gamma} \mathbb{E}\left[\Lambda_t(a)\right] \,\mathrm{d}a=\sqrt{\nu}\mathcal{B}_0\int_0^t \mathbb{E}\left[\mathbf{1}_{\Gamma}(u_j(s))\right]\,\mathrm{d}s=\sqrt{\nu}t\mathcal{B}_0 \mathbb{P}((u,e_j)_{L^2} \in \Gamma)\,.\]
Combining these two identities and remarking that $|h_j(0)|\lesssim |j|\left(\|u\|_{\dot H^{2+\delta}}^2+1\right)$ 
we see that $h_j(0)$ is of finite expectation thus $\mathbb{P}(\langle u, e_j\rangle_{L^2} \in \Gamma) \lesssim \nu ^{-1/2} \ell (\Gamma) \mathbb{E}[|h_j(0)|]$. 
This finishes the proof that there is no atom at zero for the measure $\mu_{\nu}$. Combined with the first part this gives statement (\textit{i}) for $\mu_{\nu}$. 

It remains to prove (\textit{ii}). Applying Proposition~\ref{4propFormula} to $g(x)=x$ and taking the expectation immediately gives 
\[
    \mathbb{E}\left[\mathbf{1}_{\Gamma} (\|u\|_{L^2}^2) \sum_{n \in \mathbb{Z}_0^2}|\phi_n|^2|u_n|^2\right] \leqslant \int_{\Gamma} \mathbb{E}\left[\mathbf{1}_{(a,\infty)} (\|u\|_{\dot H^{1+\delta}}^2)\|u\|_{L^2}^2\right]\,\mathrm{d}a\,,
\]
and observe that using the bound for $\mathbb{E}_{\mu_{\nu}}[\|u\|^2_{\dot H^{2+\delta}}] = \mathcal{B}_1$ we infer the bound:
\[\int_{\Gamma} \mathbb{E}\left[\mathbf{1}_{(a,\infty)} (\|u\|_{L^2}^2)\|\nabla ^{1+\delta}u\|_{L^2}^2\right]\,\mathrm{d}a \lesssim |\Gamma|\,.\]
For any $N\geqslant 1$, we introduce $\tilde{\phi}_N \coloneqq \min \{|\phi_n|, |n|\leqslant N\}$. Then we have:
\begin{align*}
    \sum_{n \in \mathbb{Z}_0^2} |\phi_n|^2|u_n|^2 & \geqslant \tilde{\phi}_N^2 \sum_{0<|n|\leqslant N} |u_n|^2 \\
    & = \tilde{\phi}_N^2 \left(\|u\|^2_{L^2}-\sum_{|n|>N}|u_n|^2\right) \\
    & \geqslant \tilde{\phi}_N^2 \left(\|u\|^2_{L^2}-N^{-2(2+\delta)} \|u\|_{\dot H^{2+\delta}}^2\right)\,,
\end{align*}
Let $\varepsilon >0$ and remark that if $\|u\|_{L^2} \geqslant \varepsilon$ and $\|u\|_{\dot H^{2+\delta}} \leqslant \varepsilon ^{-1/2}$ we have 
\[
    \sum_{n \in \mathbb{Z}_0^2} |\phi_n|^2|u_n|^2 \geqslant \tilde{\phi}_N^2(\varepsilon ^2-N^{-2(2+\delta)}\varepsilon ^{-1})\,.
\]
Let us choose $N\coloneqq N(\varepsilon) \coloneqq \left(\frac{\varepsilon^3}{2}\right)^{-\frac{1}{2(2+\delta)}}$ to get 
$\sum_{n \in \mathbb{Z}_0^2} |\phi_n|^2|u_n|^2 \geqslant \frac{1}{2}\varepsilon\tilde{\phi}_{N(\varepsilon)}^2 =: \kappa(\varepsilon)$
where $\kappa (\varepsilon)$ goes to $0$ as $\varepsilon \to 0$. Indeed $N(\varepsilon) \to \infty$ since $(\phi_n)_{n\in\mathbb{Z}^2}$ is summable.  We introduce the set 
$ \Omega_{\varepsilon}:=\{v \in L^2, \text{such that } \|v\|_{L^2} \leqslant \varepsilon \text{ or } \|v\|_{\dot H^{2+\delta}} \geqslant \varepsilon ^{-\frac{1}{2}}\}$
and remark that 
\[
    \mathbb{P}(\Omega_{\varepsilon}) \leqslant \mathbb{P}(\|u\|_{L^2} \leqslant \varepsilon) + \mathbb{P}(\|\nabla^{2+\delta} u\|_{L^2} > \varepsilon ^{-1/2}) \lesssim \varepsilon \,,
\]
thanks to the Markov inequality and (\textit{i}). Then we decompose the set $\{\|u\|_{L^2} \in \Gamma\}$ on $\Omega_{\varepsilon}$ and $\Omega_{\varepsilon}^c$ to obtain the estimate
\begin{align*}
    \mathbb{P}\left(\|u\|_{L^2} \in \Gamma\right) &= \mathbb{P}\left(\{\|u\|_{L^2} \in \Gamma\} \cap \Omega_{\varepsilon}\right) + \mathbb{P}\left(\{\|u\|_{L^2} \in \Gamma\} \cap \Omega_{\varepsilon}^c\right) \\
    & \lesssim \mathbb{P}(\Omega_{\varepsilon}) + \kappa(\varepsilon)^{-1}\mathbb{E}\left[\mathbf{1}_{\Gamma}\left(\|u\|_{L^2} \sum_{n \in \mathbb{Z}_0^2} |\phi_n|^2|u_n|^2\right)\right] \\
    & \lesssim \varepsilon + \kappa(\varepsilon)^{-1} |\Gamma|\,.
\end{align*}
Newt, we take $\varepsilon := \ell (\Gamma)$ if $\ell(\Gamma) \neq 0$ so that $\mathbb{P}\left(\|u\|_{L^2} \in \Gamma\right) \leqslant p(|\Gamma|)$ where $p(r)=C\left(r+\kappa (r)^{-1}r\right)$. Extend $p$ by $p(0)=0$ and remark that if $ |\Gamma|=0$ then we have $\mathbb{P}\left(\|u\|_{L^2} \in \Gamma\right) \lesssim \varepsilon \to 0$
which is coherent with the definition of $p$ at zero. The proof will be complete when we check that $p$ defines indeed a continuous increasing function. It is sufficient to prove that $\kappa$ defines a decreasing function, which can be seen directly on the definition of $\kappa$ and can be made continuous up to some minor modification.  
\end{proof}
\appendix

\section{Tools from stochastic analysis}
\label{4appA}

This appendix gathers some details about Itô formulas and local times for martingales.

\subsection{About Itô formulas in infinite dimension}

In this section we explain how we have applied the Itô formula without mentioning the hypotheses in the previous sections. The version of the Itô formula that we use is due to Shirikyan~\cite{shirikyan}, see also~\cite{kuksin}, Chapter~7 for a textbook presentation.  

\begin{definition}\label{4itoDef} Let us consider a Gelfand triple $(V^*,H,V)$ and a probability space $(\Omega, \mathcal{F},\mathbb{P})$. Let $(\mathcal{F}_t)_t$ be the filtration associated to identically distributed independent Brownian motions $(\beta_n(t))_{n\in\mathbb{Z}}$. Let $(e_n)_{n\in\mathbb{Z}}$ be a Hilbertian basis of $H$ and $\phi : H \to H$ a linear map. Let $y(t)$ be a $\mathcal{F}_t$ progressivly measurable process which writes
\[y(t)=y(0)+ \int_0^t x(s)\,\mathrm{d}s + \sum _{n\in\mathbb{Z}} \phi(e_n)\beta_n (t)\,.\]
We assume that $u \in \mathcal{C}^0(\mathbb{R}_+,H) \cap L^{2}_{\operatorname{loc}}(\mathbb{R}_+,V)$, $x$ is almost surely in $L^2_{\operatorname{loc}}(\mathbb{R}_+,V^*)$ and \[\sum_{n\in\mathbb{Z}} \|\phi(e_n)\|_H^2 < \infty\,.\]
Such a process is called a \textit{standard Itô process}. 
\end{definition}

\begin{proposition}[Itô formula]\label{4ito} Let $y(t)$ be a \textit{standar Itô process} as in Definition~\ref{4itoDef}. Let $F : H \to \mathbb{R}$ be twice differentiable and uniformly continuous on bounded subsets. Assume also that:
\begin{enumerate}[label=(\textit{\roman*})]
    \item Let $T>0$ and assume that there exists a continuous function $K_T$ such that for all $u\in V, v\in V^*$ there holds
    \[|\mathrm{d}F(u;v)|\leqslant K_T(\|u\|_H)\|u\|_V\|v\|_{V^*}\,.\] 
    \item If $w_k \to w$ in $V$ and $v \in V^*$ then $\mathrm{d}F(w_k;v) \to \mathrm{d}F(w;v)$. 
\end{enumerate}
Then there holds:
\[
    \mathbb{E}[F(y(t))]=\mathbb{E}[F(y(0))] + \int_0^t\mathbb{E}\left[\mathrm{d}F(y(s);x(s))+\frac{1}{2}\sum_{n \in \mathbb{Z}^2}\mathrm{d}^2F(y(t);\phi(e_n),\phi(e_n))\right]\,\mathrm{d}t\,.
\]
\end{proposition}

\begin{proof} This is a combination of Theorem~7.7.5 and the proof of Corollary~7.7.6 in~\cite{kuksin}. 
\end{proof}

The Itô formula can be applied to the process $u_{\nu}(t)$ of Section~\ref{4sec3} as a Markovian system with Gelfand triplet $(V^*,H,V):=(H^{-\delta},H^{1},H^{2+\delta})$. The functional $F(u)=\|u\|^2_{\dot{H}^1}$ is such that $\mathrm{d}F(\cdot ; \cdot)$ is continuous on $V^* \times V $. We also compute $\mathrm{d}^2F(u;v,v)=2\|v\|^2_{H^1}$, and taking into account that $\left(B(u,u),u\right)_{H^1}=0$; which is the key cancellation in dimension $2$, at the $H^1$ regularity, the Itô formula reads
\[
    \mathbb{E}\left[\|u_{\nu}(t)\|^2_{\dot H^1}\right] - \mathbb{E}\left[\|u_{\nu}(0)\|^2_{\dot H^1}\right] = \nu\int_0^t \left( \mathcal{B}_1 - 2\mathbb{E}\left[\|u_{\nu}(t')\|^2_{\dot H^{2+\delta}}\right]\right) \,\mathrm{d}t'\,.
\]

Similarly we can apply the Itô formula to $G(u):=e^{\gamma \|u\|_{\dot H^1}^2}$ in dimension $2$. We have $\mathrm{d}G(u;v)=2 \gamma \langle u,v\rangle_{H^1}e^{\gamma \|u\|_{\dot H^1}^2}$ and the second derivative is $\mathrm{d}^2G(u ;v,v)= 2\gamma \left(\|v\|^2_{\dot H^1}+2\gamma \langle u,v\rangle_{\dot H^1}^2\right)e^{\gamma \|u\|_{\dot H^1}^2}$ so that all hypotheses of Proposition~\ref{4ito} are satisfied. Keeping in mind that we have the cancellation $(u,B(u,u))_{\dot H^1}=0$, the Itô formula now reads: 
\begin{align*}\mathbb{E}\left[e^{\gamma \|u_{\nu}(t)\|_{\dot H^1}^2}\right]&=\mathbb{E}\left[e^{\gamma \|u_0\|_{\dot H^1}^2}\right] \\
    &+2\gamma \nu \mathbb{E}\left[\int_{0}^t e^{\gamma \|u_{\nu}(t')\|_{\dot H^1}^2}\left(\frac{\mathcal{B}_1}{2}-\|u_{\nu}(t')\|_{\dot H^{2+\delta}}^2 + \gamma \sum_{n\in \mathbb{Z}^2} |n|^2|\phi_n|^2 |u_n(t')|^2\right)\,\mathrm{d}t'\right]\,,
\end{align*}
where $u_n(t)=(u(t),e_n)_{L^2}$.
 
\subsection{Local times for martingales}

This appendix gathers some preliminary material used in Section~\ref{4sec6}. We start with the main abstract result on local times for martingales and explain how it applies to our purposes. 

\begin{theorem}[See \cite{KS91} Theorem~7.1 in Chapter~3]\label{4theoLocalTime} Let $y(t)$ be a  \textit{ standard Itô process} of the form \[y(t)=y(0)+\int_0^tx(s)\,\mathrm{d}s + \sum_{n\in\mathbb{Z}} \int_0^t \theta_n(s)\,\mathrm{d}\beta_n(s)\,,\]
where $x(t), \theta_n(t)$ are $\mathcal{F}_t$-adapted processes such that there holds \[\mathbb{E} \left[\int_0^t \left(|x(s)|+\sum_{n \in\mathbb{Z}} |\theta_n(s)|^2\right)\,\mathrm{d}s\right] < \infty \text{ for any } t>0\,.\]
Then there exists a random field that we denote by $\Lambda_t(a,\omega)$, for $t \geqslant 0$, $a \in \mathbb{R}$, $\omega \in \Omega$ such that the following properties hold. 
\begin{enumerate}[label=(\textit{\roman*})]
    \item $(t,a,\omega) \mapsto \Lambda_t(a,\omega)$ is measurable and for any $a \in \mathbb{R}$ the process $t \mapsto \Lambda_t(a, \cdot)$ is $\mathcal{F}_t$-adapted continuous and non-decreasing. For any $t\geqslant 0$, and almost every $\omega \in \Omega$ the function $a \mapsto \Lambda _t (a, \omega)$ is right-continuous. 
    \item For any non-negative Borel function $g : \mathbb{R} \to \mathbb{R}$ and with probability $1$ we have for any $t \geqslant 0$
    \begin{equation}
    \label{4eqlocal1}
    \int_0^tg(y(s)) \left(\sum_{n \in\mathbb{Z}} |\theta_n(s)|^2\right)\,\mathrm{d}s = 2 \int_{\mathbb{R}}\Lambda_t(a,\omega)\,\mathrm{d}a\,.
    \end{equation}
    \item For any convex function $f : \mathbb{R} \to \mathbb{R}$ and with probability $1$ holds \begin{align}
    \label{4eqlocal2}
        f(y(t))&=f(y(0))+ \int_0^t \partial ^{-}f(y(s))x(s)\,\mathrm{d}s + \int_{\mathbb{R}} \Lambda _t (a, \omega) \partial^2f(\mathrm{d}a)\\
        & + \sum_{n \in\mathbb{Z}^2} \int_0^t \partial^{-}f(y(s)) \theta_n(s) \,\mathrm{d}\beta_n (s)\,,\notag
    \end{align}
    and where $\partial ^-f$ stands for the subdifferential of $f$. 
\end{enumerate}
\end{theorem}

In order to be applied in our context, we make the following remarks: 
\begin{itemize}
    \item When~\eqref{4eqlocal2} is applied to the convex function $f : x \mapsto (x-a)_+$ we obtain 
        \begin{align*}
           (y(t)-a)_+-(y(0)-a)_+ &=\int_0^t \mathbf{1}_{[a,\infty)}(y(s))x(s)\,\mathrm{d}s + \Lambda_t(a,\omega) \\ 
            &+\sum_{n \in\mathbb{Z}} \int_0^t \mathbf{1}_{[a,\infty)}(y(s))\theta_n(s)\,\mathrm{d}\beta_n(s)\,. 
        \end{align*}
    Then, assuming that both $y(s), \theta_n(s)$ and $x(s)$ are stationary processes we deduce:
        \begin{equation}
        \label{4eqlocal4}
        \mathbb{E}\left[\Lambda_t(a,\omega)\right]=-t\mathbb{E}\left[\mathbf{1}_{[a,\infty)}(y(0))x(0)\right]\,.
        \end{equation}
    \item Let $\Gamma \subset \mathbb{R}$ be a Borelian, apply~\eqref{4eqlocal1} to $g=\mathbf{1}_{\Gamma}$ and take the expectation. It writes:
        \begin{equation}
        \label{4eqlocal3}
        \int_{\Gamma} \mathbb{E}\left[\Lambda_t(a)\right]\,\mathrm{d}a= \frac{t}{2}\mathbb{E}\left[\mathbf{1}_{\Gamma}(y(0))\left(\sum_{n\in\mathbb{Z}}|\theta_n(0)|^2\right)\right]\,.
        \end{equation}
\end{itemize}

Then we can prove the following result. 

\begin{proposition}\label{4propFormula} Let $\mu_{\nu}$ be a stationary measure for~\eqref{4eqNSv2} constructed in Section~\ref{4sec4}. For any borel set $\Gamma \subset\mathbb{R}_+$ and any function $g \in \mathcal{C}^2(\mathbb{R})$ whose second derivative has at most polynomial growth at infinity we have
\begin{align*}
    \label{4eqMainTool}
    \mathbb{E}_{\mu_{\nu}}&\left[ \int_{\Gamma} \mathbf{1}_{(a,\infty)}(g(\|u\|^2_{L^2})) \left(g'(\|u\|^2_{L^2}) \left(\frac{\mathcal{B}_0}{2}-\|\nabla ^{1+\delta} u\|_{L^2}^2\right)+g''(\|u\|^2_{L^2}) \sum_{n \in \mathbb{Z}^2} |\phi_n|^2|u_n|^2\right)\,\mathrm{d}a\right] \\
    & + \sum_{n \in \mathbb{Z}^2} |\phi_n|^2 \mathbb{E}_{\mu_{\nu}}\left[\mathbf{1}_{\Gamma}(g(\|u\|^2_{L^2}))(g'(\|u\|^2_{L^2})|u_n|)^2\right]=0\,,
\end{align*}
where $u_n=(u,e_n)_{L^2}$. 
\end{proposition}

\begin{proof} The proof of this proposition comes from the previous identities. Indeed, let the functional $F : H^1 \to \mathbb{R}$ defined by $F(u)\coloneqq g(\|u\|_{L^2}^2)$ so that \[F'(u;v)=2g'(\|u\|_{L^2}^2) (u,v)_{L^2} \text{ and } F''(u; v,v)= 2g'(\|u\|_{L^2}^2) \|v\|^2_{L^2} + 4g''(\|u\|^2_{L^2}) (u,v)^2_{L^2}\,.\] 
With such a function $g$ and the process $f(t)\coloneqq g(\|u(t)\|_{L^2}^2)$, the Itô formula now reads 
\begin{equation}
    \label{4eqfStation}
    f(t)=f(0)+\nu\int_0^t A(s)\,\mathrm{d}s + 2\sqrt{\nu}\sum_{n\in\mathbb{Z}^2} \phi_n \int_0^tg'(\|u\|_{L^2}^2)u_n\,\mathrm{d}\beta_n(s)\,,
\end{equation}
where after integration by parts,
$A(t):=g'(\|u\|^2_{L^2}) \left(\mathcal{B}_0-2\|u\|_{\dot H^{1+\delta}}^2\right)+4g''(\|u\|^2_{L^2}) \sum_{n \in \mathbb{Z}^2} |\phi_n|^2u_n^2$.

Remark that for each $n\in\mathbb{Z}^2$, the process $t \mapsto g'(\|u(t)\|_{L^2})(u(t),e_n)_{L^2}$ is stationnary. Indeed, $u(t)$ is an $H^1$ stationary process and the map $G : u \mapsto g'(\|u\|_{L^2}^2)(u,e_n)_{L^2}$ is continuous from $H^1$ to $\mathbb{C}$, thus Borelian. Then by Theorem~\ref{4theoLocalTime}, and more previsely, equation~\eqref{4eqlocal3} we have
\[\int_{\Gamma} \mathbb{E}[\Lambda _t (a)]\,\mathrm{d}a=2\nu t \sum_{n\in\mathbb{Z}^2} |\phi_n|^2 \mathbb{E}\left[\mathbf{1}_{\Gamma}(f)(g'(\|u\|_{L^2}^2)|u_n|)^2\right]\,,\]
and~\eqref{4eqlocal4} yields $\mathbb{E}[\Lambda _t (a)]=-\nu t \mathbb{E}[\mathbf{1}_{(a,\infty)}(f(0))A(0)]$
which concludes the proof. 
\end{proof}

\bibliographystyle{alpha}
\bibliography{biblio}
\end{document}